\documentclass[11pt,laterpaper]{article}

\usepackage{amsfonts, amsmath, amssymb, amscd, amsthm, graphicx, verbatim, color, hyperref, relsize}
\newtheorem{thm}{Theorem}[section]
\newtheorem{cor}[thm]{Corollary}
\newtheorem{lem}[thm]{Lemma}
\newtheorem{prop}[thm]{Proposition}
\newtheorem{prob}[thm]{Problem}
\newtheorem{q}[thm]{Question}

\theoremstyle{definition}
\newtheorem{defn}[thm]{Definition}
\theoremstyle{remark}
\newtheorem{rem}[thm]{Remark}
\newtheorem{ex}[thm]{Example}

\renewcommand{\kappa }{\varkappa}
\renewcommand{\phi}{\varphi}
\renewcommand{\d }{{\rm d} }
\newcommand{\dl}{\d_{\lambda}}
\newcommand{\G }{\Gamma (G, X\sqcup \mathcal H)}
\newcommand{\dxh }{\d_{X\cup\mathcal H}\, }
\newcommand{\Hl }{\{ H_\lambda \} _{\lambda \in \Lambda } }

\newcommand{\e }{\varepsilon }
\newcommand{\Lab}{{\bf Lab}}

\newcommand{\h}{\hookrightarrow _h }

\newcommand{\AH}{\mathcal A(H)}

\newcommand{\Ind}{{\rm Ind}}
\newcommand{\AG}{\mathcal{A}(G)}
\newcommand{\MG}{\mathcal{M}(G)}

\newcommand{\AHi}{\mathcal A (H_1)\times \cdots \times \mathcal A(H_n)}
\newcommand{\AHl}{\times_{\lambda\in \Lambda} \mathcal A (H_\lambda)}
\newcommand{\Hi}{\{ H_1, \ldots, H_n\}}

\newcommand{\diam}{{\rm diam}}
\begin{document}

\title{Extending group actions on metric spaces}
\author{C. Abbott, D. Hume, D. Osin}
\date{}
\maketitle

\begin{abstract}
We address the following natural extension problem for group actions: Given a group $G$, a subgroup $H\le G$, and an action of $H$ on a metric space, when is it possible to extend it to an action of the whole group $G$ on a (possibly different) metric space? When does such an extension preserve interesting properties of the original action of $H$? We begin by formalizing this problem and present a construction of an induced action which behaves well when $H$ is hyperbolically embedded in $G$. Moreover, we show that induced actions can be used to characterize hyperbolically embedded subgroups. We also obtain some results  for elementary amenable groups.
\end{abstract}

\tableofcontents


\section{Introduction}


\paragraph{1.1. The extension problem for group actions on metric spaces.} All actions of groups on metric spaces considered in this paper are assumed to be isometric by default. Thus, by a \emph{group action on a metric space} we mean a triple $(G,S,\phi)$, where $G$ is a group, $S$ is a metric space, and $\phi $ is a homomorphism from  $G$ to ${\rm Isom}(S)$, the group of isometries of $S$. If no confusion is possible, we omit $\phi $ from the notation and denote the action $(G,S,\phi)$ by $G\curvearrowright S$.

The main goal of this paper is to address the following natural question.

\begin{q}[Extension Problem]\label{EPintr}
Given a group $G$, a subgroup $H\le G$, and an action $H\curvearrowright R$ of $H$ on a metric space $R$, does there exist an action of $G$ on a (possibly different) metric space that extends $H\curvearrowright R$?
\end{q}

There are several ways to formalize the notion of an extension. Since our interest in this question arose in the context of geometric group theory, we accept a ``coarse" definition, which focuses on large scale invariants of groups and spaces.

Let $H$ be a group acting on metric spaces, $(R,\d_R)$ and $(S,\d_S)$. Recall that a map $f\colon R\to S$ is said to be
\begin{enumerate}
\item[--] \emph{coarsely $H$--equivariant} if for all $x\in R$, we have
\begin{equation}\label{defce}
\sup_{h\in H}\d_S(f(hx),hf(x))<\infty;
\end{equation}
\item[--] \emph{a quasi-isometric embedding} if there is a constant $C$ such that for all $x,y\in R$ we have
\begin{equation}
\label{defqi}\frac1C\d_R(x,y)-C\le \d_S(f(x),f(y))\le C\d_R(x,y)+C;
\end{equation}
if, in addition, $S$ is contained in the $\e$--neighborhood of $f(R)$ for some constant $\e$, $f$ is called a \emph{quasi-isometry}.
\end{enumerate}

\begin{defn}
Let $H$ be a subgroup of a group $G$ and let $H\curvearrowright R$ be an action of $H$ on a metric space $R$. An action $G\curvearrowright S$ of $G$ on a metric space $S$ is an \emph{extension} of $H\curvearrowright R$ if there exists a coarsely $H$--equivariant quasi-isometric embedding $R\to S$. We say that the \emph{extension problem is solvable for the pair} $H\le G$ if the answer to Question \ref{EPintr} is affirmative for every action of $H$ on a metric space.
\end{defn}

In many cases the answer is negative, the most striking example being the following.

\begin{ex}
Every countable group $H$ embeds in $Sym(\mathbb N)$, the group of all permutations of natural numbers. However it is known that every action of $Sym(\mathbb N)$ on a metric space has bounded orbits \cite{Cor}. Thus no action of $H$ with unbounded orbits can be extended to an action of $Sym(\mathbb N)$
\end{ex}

On the other hand, there are many examples of pairs $H\le G$ for which the answer is obviously affirmative.

\begin{ex} \label{ex2}
\begin{enumerate}
\item[(a)] The extension problem for a pair $H\le G$ is solvable whenever $H$ is finite. Indeed, for any action of a finite group $H$ on a metric space $R$, the trivial action of $G$ on $R$ is an extension of $H\curvearrowright R$.

\item[(b)] Let $H$ be a retract of $G$ and let $\rho\colon G\to H$ be a homomorphism such that $\rho\vert_H\equiv id_H$. It is easy to see that for every action $H\curvearrowright R$, the action of $G$ on the same metric space $R$ defined by $gx=\rho(g)x$ for all $g\in G$ and $x\in R$ is an extension of $H\curvearrowright R$. Indeed the identity map $R\to R$ is an $H$-equivariant isometry.
\end{enumerate}
\end{ex}

If the group $G$ is finitely generated, solvability of the extension problem imposes strong restrictions on the geometry of the embedding $H\le G$.

\begin{prop}[Prop. \ref{undist}]\label{undist-intr}
Let $G$ be a finitely generated group. Suppose that the extension problem is solvable for some $H\le G$. Then $H$ is finitely generated and undistorted in $G$.
\end{prop}

It is worth noting that the extension problem may not be solvable even for finite index subgroups of finitely generated groups, which are always undistorted.

\begin{ex} \label{vfree}
Let $$G=\langle a,b,t\mid t^2=1, \; t^{-1}at=b\rangle \cong F(a,b)\rtimes \mathbb Z_2$$ and $H=\langle a,b\rangle \cong F(a,b)$, where $F(a,b)$ is the free group with basis $\{ a, b\}$. Then the action of $H$ that factors through the translation action of $\langle a \rangle \cong \mathbb Z$ on $\mathbb R$ does not extend to an action of $G$. Indeed, if such an extension $G\curvearrowright S$ existed, the subgroup $\langle b\rangle $ (respectively, $\langle a\rangle $) would have bounded (respectively, unbounded) orbits in $S$. However this is impossible since $a$ and $b$ are conjugate in $G$.
\end{ex}

Proposition \ref{undist-intr} implies that solvability of the extension problem for \emph{all} subgroups of a given group is a rather rare phenomenon. For example, we prove the following in Section 4.1.

\begin{thm}\label{main1}
Let $G$ be a finitely generated elementary amenable group. Then the extension problem is solvable for all subgroups of $G$ if and only if there exists a finite index free abelian subgroup $A\le G$ and the action of $G$ on $A$ by conjugation factors through the action of $\mathbb Z/2\mathbb Z $ by inversion.
\end{thm}

We then turn to the opposite side of the group theoretic universe and study Question \ref{EPintr} for groups with hyperbolic-like properties. Our main result in this direction is the following.

\begin{thm}[Cor. \ref{cor-he}]\label{main2}
Let $G$ be a group, $H$ a hyperbolically embedded subgroup of $G$. Then the extension problem for $H\le G$ is solvable.
\end{thm}

The notion of a hyperbolically embedded subgroup was introduced by Dahmani, Guirardel, and Osin in \cite{DGO}. For the definition we refer to the next section. Examples include peripheral subgroups of relatively hyperbolic groups and maximal virtually cyclic subgroups containing a pseudo-Anosov element (respectively, a fully irreducible automorphism) in mapping class groups of closed surfaces (respectively, $Out(F_n)$), etc. For details and more examples see \cite{DGO}.

We mention one application of Theorem \ref{main2} to hyperbolic groups. Recall that a subgroup $H$ of a group $G$ is \emph{almost malnormal} if $|H\cap g^{-1}Hg|<\infty $ for all $g\in G\setminus H$.

\begin{cor}[Cor. \ref{hyp}]\label{tfhyp}
Let $G$ be a hyperbolic group.
\begin{enumerate}
\item[(a)] Suppose that $H$ is quasiconvex in $G$ and either virtually cyclic or almost malnormal. Then the extension problem is solvable for $H\le G$.
\item[(b)] Conversely, if the extension problem is solvable for a subgroup $H\le G$, then $H$ is quasiconvex.
\end{enumerate}
\end{cor}

Notice that the extension problem for the pair  $F_2\le F_2\times \mathbb Z/2\mathbb Z$, where $F_2$ is the free group of rank $2$, is solvable (see Example \ref{ex2} (b)) while $F_2$ is neither virtually cyclic nor almost malnormal in $F_2\times \mathbb Z/2\mathbb Z$. Thus the sufficient condition from part (a) is not necessary for the extension problem to be solvable. On the other hand, Example \ref{vfree} shows that the necessary condition from part (b) is not sufficient.

\paragraph{1.2. Induced actions.}
Our proof of Theorem \ref{main2} and the ``if" part of Theorem \ref{main1} is based on the construction of an \emph{induced action}. We only present it for geodesic metric spaces in this paper, although a similar theory can be developed in general settings. Restricting to geodesic spaces makes our exposition less technical and still allows us to answer Question \ref{EPintr} in the full generality since any action of any group on a metric space extends to an action of the same group on a geodesic metric space (see Proposition \ref{geo}).

Induced actions can be defined in a functorial way if we consider group actions on metric spaces up to a natural equivalence relation. In order to avoid dealing with proper classes, we fix some cardinal number $c\ge \mathfrak{c}$, where $\mathfrak{c}$ is the cardinality of the continuum and, henceforth, we assume all metric spaces to have cardinality at most $c$. All results proved in this paper remain true for every such $c$. For most applications, it suffices to take $c=\mathfrak{c}$.

\begin{defn}\label{eqdef}
 We say that two actions $G\curvearrowright R$ and $G\curvearrowright S$ of a group $G$ on metric spaces $R$ and $S$ are \emph{equivalent} (and write $G\curvearrowright R \sim G\curvearrowright S$), if there exists a coarsely $G$--equivariant quasi-isometry $R\to S$.
\end{defn}

It is easy to see that $\sim $ is indeed an equivalence relation (see Proposition \ref{equiv}). The equivalence class of an action $G \curvearrowright R$ is denoted by $[G \curvearrowright R]$; we also denote by $\AG$ the set of all equivalence classes of actions of $G$ on geodesic metric spaces (of cardinality at most $c$).

Suppose that a group $G$ is generated by a subset $X$ relative to a subgroup $H$; that is, $G=\langle X\cup H\rangle $. Associated to such a triple $(G,H,X)$ is a natural map
$$
\Ind_X\colon \AH\to \AG,
$$
called the \emph{induced action}. Our construction is especially easy to describe in the particular case when $H$ is generated by a set $Y$ and $A\in \AH$ is the equivalence class of the natural action of $H$ on its Cayley graph $\Gamma (H,Y)$. Then
$$
\Ind_X(A) =[G \curvearrowright \Gamma (G, X\cup Y)].
$$
In the general case, the definition is a bit more technical: $\Ind _X ([H \curvearrowright R])$ is defined by mixing the left action of $G$ on the coset graph of $H$ (with respect to $X$) and the given action $H\curvearrowright R$ in a natural way. This involves several additional parameters, but the resulting action is independent of them up to the equivalence introduced above. If $G$ is finitely generated modulo $H$, $\Ind_X$ turns out to be independent of the choice of a finite relative generating set $X$ and is denoted simply by $\Ind $.

Given $H\le G$, we say that $B\in \AG$ is an \emph{extension} of $A\in \AH$ if some (equivalently, any) action $G\curvearrowright S\in B$ is an extension of some (equivalently, any) $H\curvearrowright R\in A$. Obviously $\Ind_X(A)$ cannot always be an extension of $A\in \AH$, as such an extension may not exist at all. However, our next theorem shows that the induced action is, in a certain sense, the best thing we can hope for. Here we state our result for relatively finitely generated groups and refer to Theorem \ref{EP<->inc} for the general case.

\begin{thm}\label{main3}
Let $G$ be group and let $H\le G$. Assume that $G$ is finitely generated relative to $H$. Then the following conditions are equivalent.
\begin{enumerate}
\item[(a)] The extension problem for the pair $H\le G$ is solvable.
\item[(b)] For every action $A\in \AH$, $\Ind(A)$ is an extension of $A$.
\item[(c)] The subgroup $H$ is incompressible in $G$.
\end{enumerate}
\end{thm}

The notion of an incompressible subgroup is introduced in Section 3.4. One can think of it as a stronger version of the notion of an undistorted subgroup of a finitely generated group. Unlike solvability of the extension problem, the property of being incompressible can be defined for a subgroup $H$ of a group $G$ in intrinsic terms, without mentioning any actions at all. Incompressibility and a generalization of Theorem \ref{main3} are instrumental in all proofs of extendability results in our paper.

Since hyperbolicity of a geodesic space is a quasi-isometry invariant, we can define \emph{hyperbolic elements} of $\AG$ to be equivalence classes of $G$-actions on hyperbolic spaces.  The following theorem shows that the construction of the induced action behaves well for hyperbolically embedded subgroups and, moreover, it can be used to characterize hyperbolic embeddings.

\begin{thm}\label{main4}
Assume that a group $G$ is generated by a set $X$ relative to a subgroup $H$.
\begin{enumerate}
\item[(a)] If $H$ is hyperbolically embedded in $G$ with respect to $X$, then for every $A\in \AH$, $\Ind_X(A)$ is an extension of $A$; if, in addition, $A$ is hyperbolic then so is $Ind_X(A)$.
\item[(b)] Suppose that $H$ is countable and for every hyperbolic $A\in \AH$, $\Ind_X(A)$ is a hyperbolic extension of $A$. Then $H$ is hyperbolically embedded in $G$ with respect to $X$.
\end{enumerate}
\end{thm}

\begin{rem}
The countability assumption in part (b) cannot be dropped, see Example \ref{ex-c}. Also the condition that $\Ind _X(A)$ is a hyperbolic extension of $A$ for every hyperbolic $A\in \AH$ cannot be replaced with the assumption that every action of $H$ on a hyperbolic space extends to an action of $G$ on a hyperbolic space, see Example \ref{ex-d}.
\end{rem}

Theorem \ref{main4} can be applied to construct interesting actions of
acylindrically hyperbolic groups on hyperbolic spaces.  Recall that
the class of acylindrically hyperbolic groups includes mapping class
groups of closed hyperbolic surfaces, $\operatorname{Out}(F_n)$ for $n\ge 2$, groups
of deficiency at least $2$, most $3$-manifold groups, and many other
examples. It is proved in \cite{DGO} that every acylindrically hyperbolic
group $G$ contains a hyperbolically embedded subgroup isomorphic to
$F_2\times K$, where $K$ is finite and $F_2$ is free of rank $2$. Thus
one can get interesting actions of $G$ on hyperbolic spaces starting
from actions of $F_2$ and applying Theorem \ref{main4} (and Example \ref{ex2} (b)).
This idea is used in \cite{ABO} to obtain several results about hyperbolic
structures on groups.

\paragraph{Acknowledgements.} The first author was supported by NSF RTG award DMS-1502553. The second author was supported by the grant ANR-14-CE25-0004 ``GAMME'', and the NSF grant DMS-1440140 while the author was in residence at the Mathematical Sciences Research Institute in Berkeley, California, during the Fall 2016 semester. The third author was supported by the NSF grant DMS-1308961.


\section{Preliminaries}


\subsection{Metric spaces and group actions: notation and terminology}
In this paper we allow distance functions on metric spaces to take infinite values. More precisely, we extend addition and ordering from $[0, \infty )$ to $[0,\infty]$ in the natural way: $c+\infty =\infty +c =\infty $ and $c\le \infty $ for all $c\in [0, +\infty]$. Following the standard terminology, by an \emph{extended metric space} we mean a set $S$ endowed with a function $\d_S\colon S\times S\to [0, +\infty]$ that satisfies the standard axioms of a metric, where addition and ordering are extended as described above. The function $\d_S$ is called an \emph{extended metric}.

Given an (extended) metric space $S$, we always denote the (extended) metric on $S$ by $\d_S$ unless another notation is introduced explicitly. All graphs (not necessarily connected) are considered as extended metric spaces with respect to the standard combinatorial metric.

Given a path $p$ in a metric space in a metric space $S$, we denote by $p_-$ and $p_+$ the origin and the terminus of $p$, respectively. The length of $p$ is denoted by $\ell (p)$. A path $p$ in a metric space $S$ is called {\it $(\lambda , c)$--quasi--geodesic} for some $\lambda \ge 1$,
$c\ge 0$ if $$\ell (q)\le  \lambda d_S(q_-, q_+)+c$$ for any subpath $q$ of $p$.

For a subset $X$ of a group $G$, we denote by $\Gamma (G,X)$ the Cayley graph of $G$ with respect to $X$. We do not assume that $X$ generates $G$ here and therefore Cayley graphs are not necessarily connected. However we always assume that all subsets $X\subseteq G$ used to form Cayley graphs, as well as all generating sets of $G$, are symmetric. That is, given $x\in X$, we always assume that $x^{-1}\in X$. By $\d_X$ (respectively, $|\cdot |_X$) we denote the extended word metric (respectively length) on $G$ associated to a subset $X\subseteq G$. That is, $|g|_X$ is the usual word length if $g\in \langle X\rangle $ and $\infty$ otherwise; the metric is defined by $\d_X(a,b)=|a^{-1}b|_X$.

In this paper we understand properness of actions in the metric sense. That is, an action of a group $G$ on a metric space $S$ is called

\begin{enumerate}
\item[--] \emph{proper} if for every bounded subset $B\subseteq S$ the set $\{ g\in G\mid gB\cap B\ne \emptyset\}$ is finite;
\item[--] \emph{cobounded} if there exists a bounded subset $B\subseteq S$  such that $S=\bigcup_{g\in G} gB$;
\item[--] \emph{geometric} if it is proper and cobounded.
\end{enumerate}

We will need the following well-known fact about actions of finitely generated groups. We provide the proof for convenience of the reader.

\begin{lem}\label{OML}
Let $G$ be a group generated by a finite set $X$. For every action of $G$ on a metric space $S$ and every $s\in S$, there exists a constant $M$ such that, for all $g\in G$, we have
\begin{equation}\label{MS}
\d_S (s, gs) \le M|g|_X.
\end{equation}
\end{lem}

\begin{proof}
Let $$M=\max_{x\in X} \d_S (s, xs).$$ Suppose that an element $g\in G$ decomposes as $g=x_1x_2\ldots x_n$ for some $x_1, x_2, \ldots, x_n\in X$ and $n=|g|_X$. Then we have
\begin{equation}\label{qi1}
\begin{array}{rcl}
\d_S(s, gs)& \le &  \d_S (s, x_1s) + \d_S (x_1s, x_1x_2s) +\cdots + \d_S(x_1\cdots x_{n-1}s , x_1\cdots x_ns) \\&&\\&& \le \d_S (s, x_1s) + \d_S (s, x_2s) +\cdots + \d_S(s , x_ns)\le  Mn=M|g|_X.
\end{array}
\end{equation}
\end{proof}

\subsection{Hyperbolic spaces}

In this paper, we say that a geodesic metric space $S$ is \emph{$\delta $-hyperbolic} for some $\delta \ge 0$ if for every geodesic triangle $\Delta $ in $S$, every side of $\Delta $ belongs to the union of the $\delta$-neighborhood of the other two sides.

We will use two properties of hyperbolic spaces.  The first one is well-known and can be found, for example, in \cite[Theorem 1.7, p.401]{BH}.

\begin{lem}\label{qg}
For any $\delta \ge 0$, $\lambda \ge 1$, $c\ge 0$, there exists a
constant $\kappa =\kappa (\delta , \lambda , c)\ge 0$ such that
every two $(\lambda , c)$--quasi--geodesics in a $\delta $-hyperbolic space with the same endpoints belong to the closed $\kappa $-neighborhoods of each other.
\end{lem}

The next lemma is a combination of a simplified version of Lemma 10 from \cite{Ols92} and the fact that in a $\delta$-hyperbolic space every side of a geodesic  quadrilateral belongs to the closed $2\delta$-neighborhood of the other $3$ sides.

\begin{lem}\label{Ols}
Let $S$ be a subset of the set of sides of a geodesic $n$--gon $\mathcal P=p_1p_2\ldots p_n$ in a $\delta $--hyperbolic space. Assume that the total lengths of all sides from $S$ is at least $10^3 cn$ for some $c\ge 30\delta $. Then there exist two distinct sides $p_i$, $p_j$, and subsegments $u$, $v$ of $p_i$ and $p_j$, respectively, such that $p_i\in S$, $\min \{ \ell(u), \ell(v)\} > c$, and $u$, $v$ belong to $15\delta $-neighborhoods of each other.
\end{lem}

We will also make use of group actions on combinatorial horoballs introduced by Groves and Manning \cite{GM}.

\begin{defn}\label{GMhoro}
Let $\Gamma$ be any graph.
The \emph{combinatorial horoball based on $\Gamma$}, denoted
$\mathcal {H}(\Gamma)$, is the graph formed as follows:
\begin{enumerate}
\item[1)] The vertex set $\mathcal{H}^{(0)}(\Gamma )$ is $\Gamma^{(0)}\times \left( \{0\}\cup \mathbb N \right)$.
\item[2)] The edge set $\mathcal{H}^{(1)}(\Gamma )$ consists of three types of edges.  The
  first two types are called \emph{horizontal}, and the last type is
  called \emph{vertical}.
\begin{enumerate}
\item[(a)] If $e$ is an edge of $\Gamma$ joining $v$ to $w$ then there is a
  corresponding edge $\bar{e}$ connecting $(v,0)$ to $(w,0)$.
\item[(b)] If $k\in \mathbb N$ and $0<\d _{\Gamma}(v,w)\leq 2^k$, then there is a single edge
  connecting $(v,k)$ to $(w,k)$.
\item[(c)] If $k\in \mathbb N$ and $v\in \Gamma^{(0)}$, there is an edge  joining
  $(v,k-1)$ to $(v,k)$.
\end{enumerate}
\end{enumerate}
\end{defn}

By $\d_\Gamma $ and $\d_{\mathcal H(\Gamma )} $ we denote the combinatorial metrics on $\Gamma $ and $\mathcal H (\Gamma )$, respectively. The following results were proved in \cite[Theorem 3.8 and Lemma 3.10]{GM}.

\begin{lem}\label{GM}
\begin{enumerate}
\item[(a)] For every connected graph $\Gamma $, $\mathcal H (\Gamma )$ is hyperbolic.
\item[(b)] For every two vertices $a,b\in \Gamma $, there exists a geodesic $p=p_1qp_2$ between $a$ and $b$ such that $p_1$ and $p_2$ entirely consist of vertical edges and $q$ has length at most $3$.
\end{enumerate}
\end{lem}

Finally will need a well-known homological variant of the isoperimetric characterization of hyperbolic graphs. Let $\Sigma $ be a graph. Given a loop $c$ in $\Sigma $, we denote its homology class in $H_1(\Sigma , \mathbb Z_2)$ by $[c]$. Let $\ell (c)$ and $\diam (c)$ denote the length and the diameter of $c$, respectively.

\begin{prop} [Bowditch, {\cite[\S 7]{B_hyp}.}]\label{IP}
For any graph $\Sigma $ the following conditions are equivalent.
\begin{enumerate}
\item[(a)] $\Sigma $ is hyperbolic.

\item[(b)] There are some positive constants $M$, $L$ such that
if $c$ is a simple loop in $\Sigma $, then there exist loops $c_1,
\ldots , c_k$ in $\Sigma $ with $\diam(c_i)\le M$ for all $i=1,
\ldots, k$ such that
\begin{equation}\label{c}
[c]=[c_1]+\ldots +[c_k]
\end{equation}
and $k\le L\ell (c)$.
\end{enumerate}
\end{prop}

\subsection{Hyperbolically embedded subgroups}

Hyperbolically embedded subgroups were formally introduced by Dahmani, Guirardel, and the third author in \cite{DGO} although the idea goes back to the paper \cite{Osi06b}, where this notion was studied in the context of relatively hyperbolic groups.

Let $G$ be a group with a fixed collection of subgroups $\Hl $.
Given a (symmetric) subset $X\subseteq G$ such that $G$ is generated by $X$ together with the union of all $H_\lambda$'s,  we denote by $\G $ the Cayley graph of $G$ whose edges are labeled by letters from the alphabet $X\sqcup\mathcal H$, where
\begin{equation}\label{calH}
\mathcal H= \bigsqcup\limits_{\lambda \in \Lambda } H_\lambda.
\end{equation}
That is, two vertices $g,h\in G$ are connected by an edge going from $g$ to $h$ which is labeled by $a\in X\sqcup\mathcal H$ if and only if $a$ represents the element $g^{-1}h$ in $G$.

Notice that the unions in the definition above are disjoint. This means, for example, that for every $h\in H_\lambda \cap H_\mu$, the alphabet $\mathcal H$ will have two letters representing the element $h$ in $G$: one in $H_\lambda$ and the other in $H_\mu$. It can also happen that a letter from $\mathcal H$ and a letter from $X$ represent the same element of $G$.

In what follows, we think of the Cayley graphs $\Gamma (H_\lambda, H_\lambda  )$ as naturally embedded complete subgraphs of $\G $.

\begin{defn}\label{dl}
For every $\lambda \in \Lambda $, we introduce an extended  metric $\dl \colon H_\lambda \times H_\lambda \to [0, +\infty]$ as follows. Let $$\Delta_\lambda = \G \setminus E(\Gamma (H_\lambda, H_\lambda  ))$$ be the graph obtained from $\G$ by excluding all edges (but not vertices) of $\Gamma(H_\lambda, H_\lambda)$. Then for $h,k\in H_\lambda$, $\dl (h,k)$ is the length of a shortest path in $\Delta_\lambda $ that connects $h$ to $k$ (we think of $h$ and $k$ as vertices of $\G$ here). If no such a path exists, we set $\dl (h,k)=\infty $. Clearly $\dl $ satisfies the triangle inequality.
\end{defn}

\begin{defn}\label{he-def}
Let $G$ be a group, $X$ a (not necessary finite) symmetric subset of $G$. We say that a collection of subgroups $\Hl$ of $G$ is \emph{hyperbolically embedded in $G$ with respect to $X$} (we write $\Hl \h (G,X)$) if the following conditions hold.
\begin{enumerate}
\item[(a)] The group $G$ is generated by $X$ together with the union of all $H_\lambda$ and the Cayley graph $\G $ is hyperbolic.
\item[(b)] For every $\lambda\in \Lambda $, the extended metric space $(H_\lambda, \dl )$ is proper, i.e., any ball of finite radius in $H_\lambda $ contains finitely many elements.
\end{enumerate}
Further we say  that $\Hl$ is \emph{hyperbolically embedded} in $G$ and write $\Hl\h G$ if $\Hl\h (G,X)$ for some $X\subseteq G$.
\end{defn}

For details and examples we refer to \cite{DGO}. The following proposition relates the notions of a hyperbolically embedded subgroup and a relatively hyperbolic group. Readers unfamiliar with relatively hyperbolic groups may think of it as a definition.

\begin{prop}[{\cite[Proposition 4.28]{DGO}}]\label{relhypdef}
Let $G$ be a group, $\Hl$ a finite collection of subgroups of $G$. Then $G$ is hyperbolic relative to $\Hl$ if and only if $\Hl\h (G,X)$ for some finite subset $X\subseteq G$.
\end{prop}

We will need the following.

\begin{lem}[{\cite[Corollary 4.27]{DGO}}]\label{rhX}
Let $G$ be a group, $\Hl$ a collection of subgroups of $G$, $X_1$, $X_2$ two relative generating sets of $G$ modulo $\Hl$ such that $X_1\bigtriangleup X_2$ is finite. Then $\Hl \h (G, X_1)$ if and only if $\Hl\h (G, X_2)$.
\end{lem}

In Sections 4.2 and 4.3 we will use terminology and results that first appeared in \cite{O06} and \cite{O07} in the context of relatively hyperbolic groups and then were generalized to hyperbolically embedded subgroups in \cite{DGO}.

\begin{defn}\label{comp}
Let $q$ be a path in the Cayley graph $\G $. A (non-trivial) subpath $p$ of $q$ is called an \emph{$H_\lambda $--subpath}, if the label of $p$ is a word in the alphabet $H_\lambda   $. An $H_\lambda$--subpath $p$ of $q$ is an {\it $H_\lambda $--component} if $p$ is not contained in a longer $H_\lambda$--subpath of $q$; if $q$ is a loop, we require in addition that $p$ is not contained in any longer $H_\lambda $--subpath of a cyclic shift of $q$.

Two $H_\lambda $--subpaths $p_1, p_2$ of a path $q$ in $\G $ are called {\it connected} if there exists an edge
$c$ in $\G $ that connects some vertex of $p_1$ to some vertex of $p_2$ is labeled by an element of $H_\lambda$. In algebraic terms this means that all vertices of $p_1$ and $p_2$ belong to the same left coset of $H_\lambda $. A component of a path $p$ is called \emph{isolated} in $p$ if it is not connected to any other component of $p$.
\end{defn}

It is convenient to enlarge the domain of the extended metric $\dl\colon H_\lambda\times H_\lambda \to [0, \infty]$ defined above to $G\times G$ by assuming $$
\dl (f,g)\colon =\left\{
\begin{array}{ll}
\dl (f^{-1}g,1),& {\rm if}\;  f^{-1}g\in H_\lambda \\
\dl (f,g)=\infty , &{\rm  otherwise.}
\end{array}\right.
$$

Given a path $p$ in the Cayley graph $\G$, we let $p_-$, $p_+$ denote the origin and the terminal point of $p$, respectively. The following result, which is a simplified version of \cite[Proposition 4.13]{DGO}, will play a crucial role in Section 4.2.

\begin{lem}\label{ngon}
Assume that $\Hl\h (G,X)$. Then there exists a constant $D$ such that for any geodesic $n$--gon $p$ in $\G$, any $\lambda \in \Lambda$, and any isolated $H_\lambda$--component $a$ of $p$, we have $\dl (a_-, a_+)\le Dn$.
\end{lem}

We will also use an isoperimetric characterization of hyperbolically embedded subgroups, which generalizes the corresponding definition of relatively hyperbolic groups. We briefly recall all necessary definitions and refer to \cite{DGO} for more details.

Let $G$, $\Hl$, $\mathcal H$, and $X$ be as above. The group $G$ is
a quotient group of the free product
\begin{equation}
F=\left( \ast _{\lambda\in \Lambda } H_\lambda  \right) \ast F(X),
\label{F}
\end{equation}
where $F(X)$ is the free group with the basis $X$. Suppose that kernel of the natural homomorphism $F\to G$ is a normal closure of a subset $\mathcal R$ in the group $F$. For every $\lambda \in \Lambda $, we denote by $\mathcal S_\lambda
$ the set of all words over the alphabet $ H_\lambda $ that represent the identity in $H_\lambda $. Then the group $G$ has the presentation
\begin{equation}
\langle X, \mathcal H  \mid  \mathcal S\cup \mathcal R\rangle,
\label{rp}
\end{equation}
where $\mathcal S=\bigcup\limits_{\lambda\in \Lambda } \mathcal S_\lambda $. In what follows, presentations of this type are called {\it relative presentations of $G$ with respect to $X$ and $\Hl$}.

\begin{defn}
A relative presentation (\ref{rp}) is called \emph{bounded} if the lengths of words from the set $\mathcal R$ are uniformly bounded; if, in addition, for every $\lambda \in \Lambda$, the set of letters from $H_\lambda$ that appear in words $R\in \mathcal R$ is finite, the presentation is called \emph{strongly bounded}.
\end{defn}

Let $\Delta $ be a van Kampen diagram over (\ref{rp}). As usual, a $2$-cell of $\Delta $ is called an $\mathcal R$-cell (respectively, a $\mathcal S$-cell) if its boundary is labeled by a word from $\mathcal R$ (respectively $\mathcal S$).

Given a word $W$ in the alphabet $X\sqcup \mathcal H$ such that $W$
represents $1$ in $G$, there exists an expression
\begin{equation}
W=_F\prod\limits_{i=1}^k f_i^{-1}R_i^{\pm 1}f_i \label{prod}
\end{equation}
where the equality holds in the group $F$, $R_i\in \mathcal R$, and
$f_i\in F $ for $i=1, \ldots , k$. The smallest possible number
$k$ in a representation of the form (\ref{prod}) is called the
{\it relative area} \label{i-ra} of $W$ and is denoted by $Area^{rel}(W)$.

Obviously $Area^{rel}(W)$ can also be defined in terms of van Kampen diagrams. Given a diagram $\Delta $ over (\ref{rp}), we define its {\it relative area}, $Area^{rel} (\Delta )$, to be the number of $\mathcal R$-cells in $\Delta $. Then $Area^{rel}(W)$ is the minimal relative area of a van Kampen diagram over (\ref{rp}) with boundary label $W$.

\begin{defn}
A function  $f\colon \mathbb N\to \mathbb N$ is a {\it relative isoperimetric function} of (\ref{rp}) if for every $n\in \mathbb N$ and every word $W$ of length at most $n$ in the alphabet $X\sqcup \mathcal H$ representing $1$ in $G$, we have $Area^{rel} (W)\le f(n)$. Thus, unlike the standard isoperimetric function, the relative one only counts $\mathcal R$-cells.
\end{defn}

The following can be found in \cite[Theorem 4.24]{DGO}.

\begin{thm}\label{ipchar}
Let $G$ be a group, $\Hl $ a collection of subgroups of $G$, $X$ a relative generating set of $G$ with respect to $\Hl $. Then $\Hl\h (G,X)$ if and only if there exists a strongly bounded relative presentation of $G$ with respect to $X$ and $\Hl $ with linear relative isoperimetric function.
\end{thm}

\subsection{Weights, length functions, and metrics on groups}

Throughout this section, let $G$ denote a group. When dealing with metrics on $G$, it is often convenient to restrict to those metrics which take values in non-negative integers. For this reason, we formally introduce the following.

\begin{defn}[The set of integral valued left invariant metrics on a group]\label{MG}
Let $\MG$ denote the set of all left invariant metrics on $G$ taking values in $\mathbb N\cup \{ 0\}$. That is, a function $\d\colon G\times G \to \mathbb N\cup \{0\}$ belongs to $\MG$ if and only if it satisfies the following conditions.
\begin{enumerate}
\item[({\bf M$_1$})] $\d$ is a metric on $G$.
\item[({\bf M$_2$})] For all $f,g,h\in G$, $\d (fg,fh)=\d(g,h)$.
\end{enumerate}
\end{defn}

One way to define a metric on a group is through weight functions.

\begin{defn}[Weights on groups]\label{WF}
A \emph{weight function} on $G$ is a map  $w\colon G\to \mathbb N\cup \{0, \infty\}$  satisfying the following conditions for all $g\in G$.
\begin{enumerate}
\item[({\bf W$_1$})] $w (g)=0$ if and only if $g=1$.
\item[({\bf W$_2$})] $w (g)=w(g^{-1})$.
\end{enumerate}
\end{defn}

To every weight function $w$ on $G$ one can associate a function $\d_w\colon G\to \mathbb N\cup \{0, \infty\}$ by letting
\begin{equation}\label{ellw}
\d_w(f,g)=\min\left\{ \left. \sum_{i=1}^k w(x_i) \; \right|\; x_1, \ldots , x_k\in G,\;x_1\cdots x_k=f^{-1}g\right\}
\end{equation}
for every $f,g\in G$. Here the minimum is taken over all possible decompositions $f^{-1}g=x_1\cdots x_k$ of $f^{-1}g$. If the minimum in (\ref{ellw}) is attained at a decomposition $f^{-1}g=x_1\cdots x_k$, we call it a \emph{geodesic decomposition} of $f^{-1}g$.

It is easy to see that $\d_w\in \MG$ if and only if the set
$$
supp (w) =\{ g\in G\mid w(g)<\infty \}
$$
generates $G$. In this case we call $\d_w$ the \emph{ metric on $G$ associated to the weight function} $w$.

We now discuss the notion of an induced metric used later in this paper. Let $\Hl$ be a collection of subgroups of $G$, $X$ a generating set of $G$ relative to $\Hl$. That is, we assume that
\begin{equation}\label{GXH}
G=\left\langle X\cup \left(\bigcup_{\lambda\in \Lambda} H_\lambda \right)\right\rangle.
\end{equation}
Suppose that we are also given a collection $C=\{ \d_{H_\lambda}\}_{\lambda \in \Lambda}$ of metrics $\d_{H_\lambda}\in \mathcal M(H_\lambda)$.

First, for every $g\in G$, we define $$\Lambda(g)=\{ \lambda\in \Lambda \mid g\in H_\lambda\}.$$ Further, let
$$
w_{C,X}(g)=\left\{
\begin{array}{ll}
0, & {\rm if }\; g=1,\\
1, & {\rm if}\; g\in X\setminus \{ 1\},\\
\min \{ \d_{H_\lambda} (1,g) \mid \lambda\in \Lambda (g)\}, & {\rm if}\; \Lambda (g)\ne \emptyset \; {\rm and\; } g\notin X,\\
\infty,& {\rm in\; all\; other\; cases}.
\end{array}\right.
$$
Obviously $w_{C,X}$ is a weight function on $G$. Let $\d_{C,X}=\d_{w_{C,X}}$ be the metric on $G$ associated to $w_{C,X}$.

\begin{defn}\label{IMG}
In the notation introduced above, we call $\d_{C,X}\in \MG$ the \emph{metric on $G$ induced by the collection $C$ (and corresponding to a relative generating set $X$)}.
\end{defn}

\begin{rem}
Alternatively, the induced metric can be defined in the following non-constructive way. Given two metrics $\d_1, \d_2\in \MG$, we write $\d_1\preceq \d_2$ if $\d_1(f,g)\le \d_2(f,g)$ for all $f,g\in G$. Obviously $\preceq $ is a partial order on $\MG$. Then $\d_{C,X}$ is the greatest element of the subset
\begin{equation}\label{subM}
\{ \d\in \MG \mid \d (1,x) =1 \; \forall \, x\in X\setminus \{ 1\} \; {\rm and }\; \d\vert _{H_\lambda} \preceq   \d_{H_\lambda}  \forall\, \lambda \in \Lambda\}.
\end{equation}
Indeed it is clear that $\d_{C,X}$ defined above belongs to the subset described by (\ref{subM}).  The fact that it is the greatest element follows immediately from (\ref{ellw}) and the triangle inequality.
\end{rem}

We conclude with an elementary example.

\begin{ex}\label{indGS}
Let $G$ be a group and let $H$ be a subgroup of $G$. Suppose that $Y$ is a generating set of $H$ and let $\d_Y$ denote the corresponding word length. Let $C=\{ \d_Y\}$. Let $X$ be a relative generating set of $G$ with respect to $H$. Then the corresponding induced metric $\d_{C,X}$ on $G$ is $\d_{X\cup Y}$, the word metric on $G$ corresponding to the generating set $X\cup Y$.
\end{ex}


\section{The extension problem and induced actions}


Throughout the rest of the paper, we deal with collections of subgroups of a given group. Sometimes we require the collections to be finite while in other cases our arguments go through without this assumption. To make it easier for the reader to distinguish between these situations we adopt the following convention: we use $\Hi$ to denote a finite collection of subgroups and $\Hl$ to denote collections which are not necessarily finite.

\subsection{The extension problem for group actions on metric spaces}

We begin by formalizing the extension problem for group actions on metric spaces in the most general situation.
Let $G$ be a group and let $\Hl$ be a (finite or infinite) collection of subgroups of $G$.

\begin{defn}[Extension of subgroups' actions]\label{defext}
We say that an action $G \curvearrowright S$ of a group $G$ on a metric space $S$ is an \emph{extension} of a collection of actions
$\{H_\lambda \curvearrowright R_\lambda\}_{\lambda\in \Lambda}$
of subgroups $H_\lambda \le G$ on metric spaces $R_\lambda$, $\lambda \in \Lambda$,
if for every $\lambda \in \Lambda $, there exists a coarsely $H_\lambda $-equivariant quasi-isometric embedding $R_\lambda \to S$.
\end{defn}

\begin{prob}[Extension problem]\label{EP}
Given a group $G$, a collection of subgroups $\Hl$ of $G$, and a collection of actions $\{H_\lambda \curvearrowright R_\lambda\}_{\lambda\in \Lambda}$ as above, does there exist an extension of $\{H_\lambda \curvearrowright R_\lambda\}_{\lambda\in \Lambda}$ to a $G$-action on a metric space?
\end{prob}

\begin{defn}
Given a group $G$ and a collection of subgroups $\Hl$ of $G$, we say that the \emph{extension problem is solvable for $\Hl$ and $G$} if the answer to Problem \ref{EP} is affirmative for every collections of actions $\{H_\lambda \curvearrowright R_\lambda\}_{\lambda\in \Lambda}$.
\end{defn}

\begin{rem}\label{onesubgr}
If the extension problem is solvable for a collection of subgroups $\Hl$ of a group $G$, then  it is obviously solvable for every $H_\lambda \le G$. The converse is false. Indeed there is an obvious obstacle: for any two subgroups $H_1$, $H_2$ from the given collection, the given actions of $H_1\cap H_2$ on the corresponding spaces $R_1$ and $R_2$ must be equivalent for the extension to exist.
\end{rem}

As we already mentioned in the introduction, the extension problem may not be solvable even for a single subgroup. Let us discuss some further examples.

Recall that a group $G$ is called \emph{strongly bounded} if every isometric action of $G$ on a metric space has bounded orbits. It is easy to show that a countable group is strongly bounded if and only if it is finite. However, there are plenty of examples of uncountable strongly bounded groups: $Sym(\mathbb N)$, $\omega_1$--existentially closed groups, and (unrestricted) infinite powers of finite perfect groups. For details on these we refer to \cite{Cor}.

\begin{ex}
It is clear that whenever $G$ is strongly bounded and $H\le G$, no action of $H$ with unbounded orbits extends to a $G$--action.
\end{ex}

We now turn to finitely generated groups. Recall that a finitely generated subgroup $H$ of a finitely generated group $G$ is \emph{undistorted} if the inclusion $H\to G$ induces a quasi-isometric embedding $(H, \d_Y)\to (G, \d_X)$, where $\d_X$ and $\d_Y$ are word metrics associated to some (equivalently, any) finite generating sets $X$ and $Y$ of $G$ and $H$, respectively. This is obviously equivalent to the requirement that there exists a constant $C$ such that $|h|_Y\le C|h|_X$ for all $h\in H$.

\begin{lem}\label{fgH}
Let $G$ be a group generated by a finite set $X$, $H\le G$, and let $\d$ be a left invariant metric on $H$. Suppose that the left action of $H$ on the metric space $(H, \d)$ extends to an action of $G$. Then there exists a constant $K$ such that $$ \d (1,h) \le K|h|_X$$ for all $h\in H$.
\end{lem}

\begin{proof}
Let $G\curvearrowright S$ be an extension of $H\curvearrowright(H, \d)$ and let $f\colon H\to S$ be the corresponding coarsely $H$-equivariant quasi-isometric embedding. Since $f$ is a quasi-isometric embedding, there is a constant $C$ such that for every $h\in H$, we have
\begin{equation}\label{fg1}
\begin{array}{rcl}
\d (1, h) & \le & C \d_S (f(1), f(h)) + C \\ && \\ & \le & C (\d_S (f(1), hf(1)) +\d_S (hf(1), f(h)) +C\\ && \\ &\le & C \d_S (f(1), hf(1)) +CD +C,
\end{array}
\end{equation}
where $D=\sup_{h\in H} \d _S (hf(1), f(h))<\infty$ by coarse $H$-equivariance of $f$. Further by Lemma \ref{OML} applied to $s=f(1)$ we have
$$
\d_S (f(1), hf(1)) \le M |h|_X
$$
for some constant $M$ independent of $h$. Combining this with (\ref{fg1}) we obtain the required inequality.
\end{proof}

\begin{prop}\label{undist}
For any finitely generated group $G$ and any $H\le G$, the following hold.
\begin{enumerate}
\item[(a)] If the extension problem is solvable for $H\le G$, then $H$ is finitely generated.

\item[(b)] Suppose that $H$ is finitely generated. Then some (equivalently, any) geometric $H$--action on a geodesic metric space extends to a $G$--action if and only if $H$ is undistorted in $G$.
\end{enumerate}
\end{prop}

\begin{proof}
(a) Let $X$ denote a finite generating set of $G$ and let $\d_X$ be the corresponding metric on $G$. Arguing by contradiction, assume that $H$ is not finitely generated. Then we can find an infinite generating set $Y=\{ y_i\mid i\in \mathbb N\}$ of $H$ with the property that
\begin{equation}\label{xi+1}
y_{i+1}\notin \langle y_1, \ldots, y_i\rangle
\end{equation}
for all $i\in \mathbb N$.
We choose any increasing sequence $(M_i)\subseteq \mathbb N$ such that
\begin{equation}\label{LiMi}
\lim_{i\to \infty} M_i/|y_i|_X=\infty
\end{equation}
and define a weight function on $G$ by the rule
$$
w(g)= \left\{ \begin{array}{ll}
M_i, & {\rm if}\; g=y_i \; {\rm for \; some\;} i\in \mathbb N,\\
\infty, & {\rm otherwise}.\end{array}\right.
$$
Let $\d _w$ be the associated metric on $H$. By our assumption, the left action of $H$ on itself endowed with the metric $\d_w$ extends to an action of $G$. Therefore, by Lemma \ref{fgH} we have  $\d_w (1, h) \le K |h|_X$ for some constant $K$ independent of $h$. In particular, we have $\d_w (1, y_i)\le K|y_i|_X=o(M_i)$ as $i\to \infty$ by (\ref{LiMi}). On the other hand, we have $\d_w(1, y_i)\ge M_i$ for all $i$; indeed every decomposition of $y_i$ into a product of elements of $Y$ must contain an element of weight at least $M_i$ by (\ref{xi+1}). A contradiction.

(b) Let $X$ (respectively, $Z$) be a finite generating set of $G$ (respectively, $H$).  If $H$ is undistorted in $G$, it is straightforward to verify that  $G\curvearrowright (G, \d_X)$ is an extension of $H\curvearrowright (H, \d_Z)$. By the Svarc-Milnor Lemma (see Proposition 8.19 in Chapter I.8 of \cite{BH}), every geometric $H$-action $H\curvearrowright R$ on a geodesic space $R$ is equivalent to $H\curvearrowright (H, \d_Z)$; therefore, $G\curvearrowright (G, \d_X)$ is an extension of $H\curvearrowright R$ as well. Conversely, suppose that $H\curvearrowright (H,\d_Z)$ extends to a $G$--action. By Lemma \ref{fgH}, there exists a constant $K$ such that $|h|_Z\le K|h|_X$ for all $h\in H$. Thus $H$ is undistorted in $G$.
\end{proof}

\begin{ex}
Let $$G=\langle a,b \mid b^{-1}ab=a^2\rangle .$$ Then the extension problem for $H=\langle a \rangle $ and $G$ is unsolvable. Moreover, the natural action of $H\cong \mathbb Z$ on $\mathbb R$ does not extend to an action of $G$. Indeed it is well-known and easy to see that $H$ is (exponentially) distorted in $G$.
\end{ex}

Recall that solvability of the extension problem does not pass to finite index subgroups (see Example \ref{vfree}). However we have the following result, which will be used in the proof of Corollary \ref{tfhyp}.

\begin{lem}\label{finind}
Let $G$ be a group, $H$ a subgroup of $G$, $K$ a finite index subgroup of $H$. Suppose that the extension problem is solvable for $K\le G$. Then it is solvable for $H\le G$.
\end{lem}

\begin{proof}
Let $H\curvearrowright R=(H, R, \alpha)$ be an action of $H$ on a metric space $R$ and let $K\curvearrowright R= (K, R, \alpha\vert_K)$. Let $G\curvearrowright S$ denote an extension of $K \curvearrowright R$ and let $f\colon R\to S$ be a coarsely $K$--equivariant quasi-isometric embedding. Fix some $r\in R$.  Let $T$ be a finite subset of $H$ such that $H=KT$. Further let $$M=\max_{t\in T} \d_S(f(tr), tf (r))$$ and $$C=\max_{t\in T}\sup_{k\in K} \d_S (f(ktr), kf(tr)).$$ Since $f$ is coarsely $K$--equivariant, $C$ is finite. For every $h\in H$, we have $h=kt$ for some $t\in T$ and $k\in K$. Therefore,
$$
\d_S (f(hr), hf(r))\le \d_S (f(ktr), kf(tr))+ \d_S (kf(tr), ktf(r)) \le C +\d_S(f(tr), tf(r)) \le C+M.
$$
Thus $f $ is also coarsely $H$--equivariant and hence $G\curvearrowright S$  is an extension of  $H\curvearrowright R$.
\end{proof}

We conclude this section with an elementary observation which allows us to reduce the extension problem to the particular case of group actions on geodesic metric space.

\begin{prop}\label{geo}
Let $G$ be a group acting on a metric space $R$. Then there exists an extension $G\curvearrowright S$ of $G\curvearrowright R$ such that $S$ is geodesic.
\end{prop}

\begin{proof}
Let $S$ be the complete graph with the vertex set $V(S)=R$. For every $x,y\in R$ such that $x\ne y$, we identify the edge connecting $x$ and $y$ with a segment of length $\d_R (x,y)$. This induces a metric on $S$ in the obvious way. The action of $G$ on $V(S)=R$ extends (in the usual sense) to an isometric action on $S$ and it is straightforward to check that the identity map $R\to R=V(S)$ induces a $G$-equivariant isometric embedding $R\to S$.
\end{proof}

\subsection{Equivalence of group actions}

We begin by showing that the relation $\sim$ introduced in Definition \ref{eqdef} is indeed an equivalence relation. This is fairly elementary and straightforward to prove. Nevertheless, we decided to provide a complete proof since this concept is central to our paper.

\begin{defn}\label{defn:cinv}
Let $X$ and $Y$ be two metric spaces and let $\alpha\colon X\to Y$ be a map. We say that a map $\beta \colon Y\to X$ is a (right) \emph{coarse inverse} of $\alpha$ if
\begin{equation}\label{betay}
\d_Y(\alpha \circ\beta (y), y)\le \e
\end{equation}
for all $y\in Y$.
\end{defn}

\begin{lem}\label{lem:cinv}
Let $G\curvearrowright X$ and $G\curvearrowright Y$ be two actions of a group $G$ on metric spaces. Let $\alpha \colon  X\to Y$ be a coarsely $G$--equivariant quasi-isometry. Then there exists a coarse inverse $\beta \colon Y\to X$ of $\alpha$ and every coarse inverse of $\alpha $ is a coarsely $G$--equivariant quasi-isometry.
\end{lem}

\begin{proof}
Since $\alpha $ is a quasi-isometry, $Y$ coincides with the $\e$--neighborhood of $\alpha (X)$ for some constant $\e$. Therefore, to every $y\in Y$ we can associate a point $\beta (y)\in X$ such that (\ref{betay}) holds. Thus we get a map $\beta \colon Y\to X$, which is a coarse inverse of $\alpha$.

Let now $C$ denote the quasi-isometry constant of $\alpha$ (as defined in the Introduction) and let $\beta \colon Y\to X$ be any map satisfying (\ref{betay}).  Given $y\in Y$, let $$ D=\sup_{g\in G} \d_Y(\alpha(g\beta(y)), g\alpha\circ\beta(y)).$$ Then for every $g\in G$, we have
$$
\begin{array}{rcl}
\d_X(\beta(gy), g\beta(y))& \le & C(\d _Y(\alpha\circ \beta(gy),\alpha(g\beta(y))) +C)\le \\&&\\&& C(\d _Y(\alpha\circ \beta(gy),g\alpha\circ\beta(y)) +C+D)\le \\&&\\&&
C(\d _Y(\alpha\circ \beta(gy),gy) +\d_Y(gy,g(\alpha\circ \beta(y))) +C+D)\le  C(2\e +C+D).
\end{array}
$$
Thus $\beta $ is coarsely $G$--equivariant.

Further,  for every $y_1, y_2\in Y$, we have
$$
\d_X(\beta (y_1), \beta(y_2)) \le C\left(\d_Y(\alpha \circ \beta (y_1), \alpha \circ\beta (y_2))+C\right)\le C\left( \d_Y(y_1, y_2)+ 2\e+C\right).
$$
Similarly, we obtain
$$
\d_X(\beta (y_1), \beta(y_2)) \ge \frac1C\left(\d_Y(\alpha \circ \beta (y_1), \alpha \circ\beta (y_2))-C\right)\ge \frac1C\left( \d_Y(y_1, y_2)- 2\e-C\right).
$$
Finally, we note that $X$ belongs to the closed $C(\e +C)$--neighborhood of $\beta (Y)$. Indeed for every $x\in X$, we have
$$
\d_X(x, \beta\circ \alpha (x))\le C\left( \d_Y(\alpha (x), \alpha \circ \beta \circ \alpha (x))+C\right)\le C(\e +C)
$$
by (\ref{betay}). Thus $\beta $ is a quasi-isometry.
\end{proof}

\begin{prop}\label{equiv}
The relation $\sim$ introduced in Definition \ref{eqdef} is an equivalence relation.
\end{prop}

\begin{proof}
Reflexivity and transitivity are obvious. That $\sim $ is symmetric follows from Lemma \ref{lem:cinv}.
\end{proof}

Recall that given a connected graph $\Gamma$, we can think of it as a metric space with respect to the standard combinatorial metric (that is, the metric obtained by identifying every edge of $\Gamma $ with $[0,1]$). The fact stated below will be used in the next section.

\begin{prop}\label{graph}
For any group $G$, any action of $G$ on a geodesic metric space is equivalent to an action of $G$ on a graph (endowed with the combinatorial metric) with trivial vertex stabilizers.
\end{prop}

\begin{proof}
Suppose that $G$ acts on a geodesic metric space $S$. Let $\Gamma $ be the graph with vertex set $V(\Gamma )=G\times S$ and the set of edges consisting of all pairs $\{ (g_1, s_1), (g_2, s_2) \}\subseteq G\times S$ such that $\d_S (s_1,s_2)\le 1$. The given action of $G$ on $S$ and the left action of $G$ on itself extend (in the usual sense) to an action on $V(\Gamma )$, which in turn can be extended to a $G$--action on $\Gamma$ since we define edges in a $G$-equivariant way. It is straightforward to verify that the action of $G$ on $V(\Gamma)$ is free and the map $s\to (1,s)$ induces a coarsely $G$-equivariant quasi-isometry $S\to \Gamma$.
\end{proof}

\begin{rem}
Note however that the action of $G$ on $\Gamma $ constructed in Proposition \ref{graph} may not be free as edges may have non-trivial (setwise) stabilizers generated by involutions. We could make the action of $G$ free by doubling these edges, but having free action on the vertex set is sufficient for our goals.
\end{rem}

\begin{defn}
Given a group $G$, we denote by $\AG$ the collection of all equivalence classes of actions of $G$ on geodesic spaces (of cardinality at most $c$).
\end{defn}

\begin{defn}
Given a group $G$, a collection of subgroups $\Hl$, and $A\in \AG$, $(B_\lambda)_{\lambda\in \Lambda }\in \AHl$ we say that $A$ is an \emph{extension} of $B$ if every action $\mathcal A\in A$ is an extension of every action $\mathcal B_\lambda \in B_\lambda$ for all $\lambda \in \Lambda$.
\end{defn}

The following proposition allows us to replace `every' with `some' in the definition above.

\begin{prop}\label{AABB}
Let $G$ be a group and $\Hl$ a collection of subgroups of $G$. Let $A\in \AG$ and $(B_\lambda)_{\lambda\in \Lambda }\in \AHl$. Then $A$ is an extension of $B$ if and only if for all $\lambda \in \Lambda$, some action $\mathcal A\in A$ is an extension of some action $\mathcal B_\lambda \in B_\lambda$.
\end{prop}

\begin{proof}
This follows immediately from the fact that a composition of a quasi-isometry and a quasi-isometric embedding (in any order) is again a quasi-isometric embedding.
\end{proof}

\subsection{Induced action}

We are now ready to introduce the concept of an induced action. Throughout this section, let $G$ be a group, $\Hi$ a finite collection of subgroups of $G$, and let $X$ be a generating set of $G$ relative to $\Hi$. Proposition \ref{geo} and Proposition \ref{graph} imply that in order to solve the extension problem, it suffices to deal with actions on connected graphs whose restrictions to vertex sets are free. Henceforth, we fix a collection of actions
$$
\mathcal A = \{ H_1\curvearrowright R_1, \ldots, H_n\curvearrowright R_n\}
$$
on graphs $R_1, \ldots , R_n$ such that the restrictions of these actions to the vertex sets $V(R_1)$, \ldots, $V(R_n)$ are free.

We denote by $\Gamma (G,X)$ the Cayley graph of $G$ with respect to $X$. Notice that $\Gamma (G,X)$ is not necessarily connected since $X$ may not generate $G$ by itself.

Roughly speaking, the induced action of $G$ associated to these data is the natural action of $G$ on the space obtained from $\Gamma (G,X)$ by gluing copies of $R_i$ to all cosets $gH_i$ along a fixed $H_i$--orbit in $R_i$. The construction will depend on the choice of coset representatives of $H_i$ and particular orbits in each $R_i$. However, all induced actions constructed in this way will be equivalent; finiteness of the collection of subgroups will be essentially used in proving this.

To define the induced action formally, we fix a \emph{collection of base vertices}
$$
\mathcal B=\{ b_1, \ldots, b_n\},
$$
where  $b_i\in V(R_i)$. For each $i$ we also fix a collection $T_i$ of representatives of left cosets of $H_i$ in $G$ and denote by $t_i\colon G\to T_i$ the map assigning to an element $g\in G$, the representative of the coset $gH_i$. Without loss of generality, we can (and will) assume that
\begin{equation}\label{t1}
t_i(h)=1 \;\; \forall \; h\in H_i.
\end{equation}
Let
$$
\mathcal T=\{ t_1, \ldots, t_n\}.
$$
We call $\mathcal T$ the \emph{transversal} of $G$ with respect to $\{ H_1, \ldots, H_n\}$.

For each $i=1, \ldots, n$, let
\begin{equation}\label{Yi}
Y_i = G/H_i \times R_i.
\end{equation}
We think of $Y_i$ as a graph, which is a disjoint union of copies $\{ gH_i\} \times R_i$ of $R_i$, for all $gH_i\in G/H_i$. We endow every $Y_i$ with the combinatorial metric (which may take infinite values as $Y_i$ is not connected unless $H_i=G$), so every $Y_i$ becomes an extended metric space.

We first want to define an action of $G$ on each $Y_i$ such that:
\begin{enumerate}
\item[({\bf A}$_1$)] It extends the action of $H_i$ on $R_i$ ($R_i$ is identified with $\{H_i\}\times R_i$) in the set theoretic sense, i.e., $h(H_i, r)=(H_i,hr)$ for all $h\in H_i$ and $r\in R_i$.
\item[({\bf A}$_2$)] $G$ permutes subsets $\{ aH_i\} \times R_i$ according to the rule $g(\{ aH_i\} \times R_i)=\{ gaH_i\} \times R_i$ for all $g,a\in G$.
\end{enumerate}

It is fairly easy to see that there is a unique way to define such an action of $G$. For every $i\in \{ 1, \ldots, n\}$, $r\in R_i$, and $g,a \in G$, we let
\begin{equation}\label{ind1}
 g(aH_i,r)=(gaH_i, \alpha_i(g,a)r)
\end{equation}
where
\begin{equation}\label{ind2}
\alpha_i(g,a)=(t_i(ga))^{-1}gt_i(a).
\end{equation}
Note that $\alpha_i(g,a)r\in R_i$ is well-defined since $\alpha_i (g,a)\in H_i$.

\begin{lem}\label{act}
Formulas (\ref{ind1}) and (\ref{ind2}) define an isometric action of $G$ on each $Y_i$ satisfying conditions ({\bf A}$_1$) and ({\bf A}$_2$). The restriction of this action to each vertex set $V(Y_i)$ is free.
\end{lem}

\begin{proof}
Throughout the proof we fix some $i\in \{ 1, \ldots, n\}$. First let us check that the identity element acts trivially. We obviously have
$\alpha_i(1,a)=(t_i(a))^{-1}t_i(a)=1$ and hence $1(aH_i,r)=(aH_i, \alpha_i(1,a)r)=(aH_i,r)$. Further, using (\ref{ind2}) we obtain
\begin{equation}\label{afga}
\begin{array}{rcl}
\alpha _i(fg,a)& = & (t_i(fga))^{-1}fgt_i(a)= (t_i(fga))^{-1}ft_i(ga) ((t_i(ga))^{-1}gt_i(a))=\\
&& \\&& \alpha _i(f,ga)\alpha _i(g,a)
\end{array}
\end{equation}
for all $f,g,a\in G$. Therefore,
$$
\begin{array}{rcl}
(fg)(aH_i, r)& = & (fgaH_i, \alpha_i(fg,a)r)=(fgaH_i, \alpha _i(f,ga)\alpha _i(g,a)r)=\\ &&\\&&f(gaH_i,\alpha_i(g,a)r)=f(g(aH_i, r)).
\end{array}
$$
Thus formulas (\ref{ind1}) and (\ref{ind2}) indeed define an action of $G$.

That the restriction of the action to $V(Y_i)$ is free easily follows from our assumption that the action $H_i\curvearrowright V(R_i)$ is free. Indeed assume that we have
\begin{equation}\label{gfix}
g(aH_i, r)=(aH_i, r)
\end{equation}
for some $(aH_i, r)\in V(Y_i)$. Without loss of generality we can assume that $a\in T_i$, i.e., $t_i(a)=a$.  Comparing (\ref{gfix}) to (\ref{ind1}), we obtain $ga\in aH_i$ and hence $t_i(ga)=a$. Together with (\ref{ind2}), this implies
$\alpha_i(g,a)= a^{-1}ga$. Again combining (\ref{gfix}) and (\ref{ind1}), we obtain $a^{-1}gar=r$. Since the action of $G$ on $V(R_i)$ is free, we have $a^{-1}ga=1$ and hence $g=1$.

If $h\in H_i$, we obtain $\alpha_i(h,1)=(t_i(h))^{-1}ht_i(1)=h$ using (\ref{t1}). This implies ({\bf A}$_1$). Condition ({\bf A}$_2$) follows from (\ref{ind1}) immediately.

It remains to show that $G\curvearrowright Y_i$ is isometric. Let $y_1=(a_1H_i, r_1)$ and $y_2=(a_2H_i, r_2)$ be two points of $Y_i$. First assume that $a_1H_i\ne a_2H_i$. Then $\d_{Y_i} (y_1, y_2)=\infty $. By (\ref{ind1}) we also have $\d_{Y_i} (gy_1, gy_2)=\infty$ as $ga_1H_i\ne ga_2H_i$ in this case. Further if $a_1=a_2=a$, then $$ \d_{Y_i} (gy_1, gy_2)= \d_{R_i} (\alpha_i(g,a)r_1, \alpha_i(g,a)r_2)=\d_{R_i} (r_1, r_2)=\d_{Y_i} (y_1, y_2)$$ by (\ref{ind1}) and the definition of the extended metric on $Y_i$.
\end{proof}

\begin{defn}[The induced action]\label{def-ia}
Let $S_{X, \mathcal T, \mathcal B, \mathcal A}$ denote the graph obtained by gluing $Y_i$, $i=1, \ldots, n$, to $\Gamma(G,X)$ by identifying vertices $g(H_i,b_i)\in V(Y_i)$ and $g\in V(\Gamma (G, X))$ for all $g\in G$ and $i=1, \ldots, n$. We call the graph $S_{X, \mathcal T, \mathcal B, \mathcal A}$ the \emph{space of the induced action}.

Since $S_{X, \mathcal T, \mathcal B, \mathcal A}$ is obtained by gluing, vertices of $S_{X, \mathcal T, \mathcal B, \mathcal A}$ are, formally speaking, equivalence classes. By abuse of notation, we will use representatives of these equivalence classes to denote vertices of $S_{X, \mathcal T, \mathcal B, \mathcal A}$. Thus we think of a vertex of $S_{X, \mathcal T, \mathcal B, \mathcal A}$ is a pair $(gH_i, v)$, where $g\in G$ and $v\in V(R_i)$ for some $i$. If $v$ does not belong to the $H_i$-orbit of $b_i$, then the equivalence class of $(gH_i, v)$ consists of a single pair. Otherwise the corresponding vertex has exactly $n$ representatives. Indeed, in the graph $S_{X, \mathcal T, \mathcal B, \mathcal A}$, we have
\begin{equation}\label{ec}
(gH_i, t_i(g)^{-1}gb_i) = g(H_i, b_i) = g(H_j, b_j)= (gH_j, t_j^{-1}(g)gb_j)
\end{equation}
for all $i,j\in \{ 1, \ldots , n\}$.

Since the gluing maps used to construct $S_{X, \mathcal T, \mathcal B, \mathcal A}$ are $G$-equivariant, the actions of $G$ on graphs $Y_i$ and $\Gamma (G,X)$ induce an action of $G$ on $S_{X, \mathcal T, \mathcal B, \mathcal A}$, denoted by $\Ind_{X, \mathcal T, \mathcal B} (\mathcal A)$.
\end{defn}

Since the actions $H_i\curvearrowright V(Y_i)$ are free by Lemma \ref{act}, the natural ($G$-equivariant) maps $Y_i\to S_{X, \mathcal T, \mathcal B, \mathcal A}$ and $\Gamma (G,X)\to S_{X, \mathcal T, \mathcal B, \mathcal A}$ are injective. Henceforth, we will think of $\Gamma (G,X)$ and $Y_i$, $i=1, \ldots, n$, as subgraphs of $S_{X, \mathcal T, \mathcal B, \mathcal A}$.

\begin{lem}\label{SisMS}
The graph $S_{X, \mathcal T, \mathcal B, \mathcal A}$ is connected.
\end{lem}
\begin{proof} By construction, every vertex $(gH_i,v)$ in $S_{X, \mathcal T, \mathcal B, \mathcal A}$ can be connected to the vertex $g=g(H_i, b_i)$ of $\Gamma(G,X)\subseteq S_{X, \mathcal T, \mathcal B, \mathcal A}$ inside $\{ gH_i\} \times R_i\subseteq Y_i$. Thus it suffices to show that every two vertices of $\Gamma(G,X)$ can be connected by a path in $S_{X, \mathcal T, \mathcal B, \mathcal A}$

Let $y=y(H_i,b_i)$ and $z=z(H_i, b_i)$ be any two vertices of $\Gamma(G,X)$. Since $X\cup\left(\bigcup_{i=1}^n H_i\right)$ generates $G$, we have
$$
z=yx_1h_1x_2h_2\ldots x_mh_m,
$$
where $x_j\in X\cup\{1\}$ and $h_j\in H_{i(j)}$ for each $j$. By construction, any two vertices of the form $g=g(H_i,b_i)$ and $gx_j=gx_j(H_i,b_i)$ are connected by an edge of $\Gamma(G,X)$, while any two vertices of the form $g=g(H_i,b_i)= g(H_{i(j)}, b_{i(j)})$ and $gh_j=gh_j(H_i,b_i)=gh_j(H_{i(j)}, b_{i(j)})$ are connected by a path in the subgraph $\{gH_{i(j)}\}\times R_{i(j)}$ of $S_{X, \mathcal T, \mathcal B, \mathcal A}$. Thus, there is a path in $S_{X, \mathcal T, \mathcal B, \mathcal A}$ connecting $y$ to $z$.
\end{proof}

We are now ready to prove the main result of this section.

\begin{prop}\label{prop-ind}
Let $G$ be a group and let $\{ H_1, \ldots, H_n\}$ be a collection of subgroups of $G$. Let $X$, $X^\prime$ be generating sets of $G$ relative to $\Hi$ such that  $|X\bigtriangleup X^\prime|<\infty$.
Let
$$
\mathcal T=\{ t_1, \ldots, t_n\}\;\;\; {\rm and}\;\;\; \mathcal T^\prime =\{ t_1^\prime, \ldots, t_n^\prime\}
$$
be transversals of $G$ with respect to $\Hi$. Let
$$
\mathcal A =\{ H_1\curvearrowright R_1, \ldots, H_n\curvearrowright R_n\}\;\;\; {\rm and}\;\;\; \mathcal A^\prime =\{ H_1\curvearrowright R_1^\prime, \ldots, H_n\curvearrowright R_n^\prime\}
$$
be collections of actions on graphs $R_i$ and $R_i^\prime$ such that the restrictions of these actions to vertex sets are free and $H_i\curvearrowright R_i \sim H_i\curvearrowright R_i^\prime$ for all $i$. Finally let
$$
\mathcal B=\{ b_1, \ldots, b_n\} \;\;\; {\rm and}\;\;\; \mathcal B^\prime=\{ b_1^\prime, \ldots, b_n^\prime\}
$$
be collections of base vertices in graphs $R_i$ and $R^\prime_i$, respectively. Then
\[\Ind_{X,\mathcal T,\mathcal B}(\mathcal A)\sim \Ind_{X^\prime,\mathcal T^\prime,\mathcal B^\prime}(\mathcal A^\prime).\]
\end{prop}

\begin{proof}
We fix coarsely $H_i$-equivariant quasi-isometries $\rho_i\colon R_i\to R^\prime_i$. Since changing $\rho_i$ by a bounded function does not violate the property of being a coarsely $H_i$-equivariant quasi-isometry, we can assume that $\rho_i$ maps $V(R_i)$ to $V(R_i^\prime)$ and
\begin{equation} \label{ribi}
\rho_i(b_i)=b_i^\prime
\end{equation}
without loss of generality. We also fix a constant $C$ such that the following inequalities hold for all $i=1, \ldots, n$:
\begin{equation}\label{rihbi}
\sup_{h\in H_i} \d_{R_i^{\prime}}(\rho_i(hb_i), h\rho_i(b_i))\le C,
\end{equation}
and
\begin{equation}\label{rixy}
\d_{R_i^{\prime}}(\rho_i(x), \rho_i(y))\le C\d_{R_i}(x,y)+C.
\end{equation}

To prove the proposition, we will construct coarsely $G$--equivariant quasi-isometries between the vertex sets
\[
V_1=V(S_{X,\mathcal T,\mathcal B,\mathcal A})\xrightarrow{\phi_1} V_2=V(S_{X,\mathcal T^\prime,\mathcal B,\mathcal A}) \xrightarrow{\phi_2} V_3=V(S_{X,\mathcal T^\prime,\mathcal B^\prime,\mathcal A^\prime}) \xrightarrow{\phi_3} V_4=V(S_{X^\prime,\mathcal T^\prime,\mathcal B^\prime,\mathcal A^\prime}).
\]
We assume that every $V_i$ is equipped by the metric induced from the corresponding graph; this metric is denoted by $\d_i$.

To change the transversals we use the map defined by
\[\phi_1(aH_i,v)=(aH_i,t^\prime_i(a)^{-1}t_i(a)v)\]
for all $a\in G$, $i=1, \ldots, n$, and $v\in V(R_i)$. Notice that this map is well-defined. Indeed if a point of $S_{X,\mathcal T,\mathcal B,\mathcal A}$ has more than one representative of type $(aH_i,v)$, then these representatives must be of the form (\ref{ec}) and we have
$$
\phi_1(gH_i, t_i(g)^{-1}gb_i)=(gH_i, t_i^\prime (g)^{-1}gb_i)=(gH_j, t_j^\prime (g)^{-1}gb_j)=\phi_1(gH_j, t_j(g)^{-1}gb_j)
$$
for all $i,j\in \{ 1,\ldots, n\}$.

To change the actions and basepoints, we define
\[
\phi_2(aH_i,v)=\left\{\begin{array}{ll} (aH_i,\alpha^\prime_i(a,1)b^\prime_i), & \textrm{if } v=\alpha_i^\prime (a,1)b_i, \\ (aH_i,\rho_i(v)), & \textrm{otherwise,}\end{array}\right.
\]
where $\alpha_i ^\prime $ is defined in the same way as $\alpha_i$ using $\mathcal T^\prime$ instead of $\mathcal T$; i.e., $\alpha_i^\prime (a,b)=t_i^\prime(ab)^{-1}at_i^\prime(b)$. Finally, to change relative generating sets we use the map
defined by $$\phi_3(aH_i,v)=(aH_i,v).$$
As above, it is straightforward to verify that $\phi_2$ and $\phi_3$ are well-defined.

The proof of the proposition is based on two lemmas.

\begin{lem} \label{fGeq}
The maps $\phi_1$, $\phi_2$, and $\phi_3$ are coarsely $G$-equivariant.
\end{lem}

\begin{proof}
Indeed, $\phi_1$ is $G$-equivariant:
\begin{align*}
g(\phi_1(aH_i,v))
&
=
g(aH_i,t^\prime_i(a)^{-1}t_i(a)v) = (gaH_i,\alpha^\prime_i(g,a)t^\prime_i(a)^{-1}t_i(a)v)
\\
&
= (gaH_i, t^\prime_i(ga)^{-1}gt_i(a)v) = \phi_T (g(aH_i,v)).
\end{align*}
We next check $\phi_2$. Fix some $g\in G$. First suppose that $x=a(H_i,b_i)$ for some $a\in G$ and $i\in\{1,\ldots,n\}$; then using (\ref{afga}) we obtain
$$
\phi_2(gx) =\phi_2(gaH_i, \alpha^\prime(ga,1)b_i) =(gaH_i, \alpha^\prime(ga,1)b_i^\prime)=g(aH_i, \alpha^\prime(a,1)b_i^\prime).
$$
In particular, $\phi_2$ is equivariant on vertices of $\Gamma(G,X)$.
If $x\notin V(\Gamma(G,X))$, then  we have  $g\phi_2(x)=(gaH_i,\alpha_i^\prime(g,a)\rho_i(v))$ and $\phi_2(gx)=(gaH_i,\rho_i(\alpha_i^\prime(g,a)(v)))$. Hence
$$
\sup_{g\in G}\d_3(\phi_2(gx),g\phi_2(x))\leq \sup_{g\in G} \d_{R_i}(\alpha_i^\prime(g,a)\rho_i(r),\rho_i(\alpha_i^\prime (g,a)v)))<\infty
$$
since the maps $\rho_i$ are coarsely $H_i$-equivariant.

Finally, for $\phi_3$ there is nothing to prove, its $G$-equivariance is obvious.
\end{proof}

\begin{lem} \label{fLip}
The maps $\phi_1$, $\phi_2$, and $\phi_3$ are Lipschitz.
\end{lem}

\begin{proof}
In each case it suffices to verify that there exists a constant $K$ such that if two vertices $x,y\in V_j$ span an edge in the corresponding graph,
then
\begin{equation}\label{di+1}
\d_{j+1}(\phi_j(x),\phi_j(y))\le K
\end{equation}
for $j=1,2,3$.

We consider two cases.

{\it Case 1}. Assume that $x$ and $y$ are connected by an edge of $\Gamma (G,X)\subseteq S_j$. Then we have $x=f(H_i,b_i)$, $y=g(H_i,b_i)$, and $f^{-1}g\in X$.  It is easy to see that $K=1$ works for $i=1,2$ by equivariance and for $i=3$ we can take $$K=\sup_{x\in X}{\d_4(f(H_i,b_i),fx(H_i,b_i))}= \sup_{x\in X}{\d_4((H_i,b_i),x(H_i,b_i))}<\infty $$ as $\d_4((H_i,b_i),x(H_i,b_i))\le 1$ for all $x\in X^\prime$ and $|X\bigtriangleup X^\prime|<\infty$.

{\it Case 2.} Next we assume that $x$ and $y$ are connected by an edge of $Y_i$ for some $i$. Then $x=(aH_i,u)$, $y=(aH_i,v)$ and $u,v$ span an edge of $R_i$. In this case it is straightforward to see that $K=1$ works for $i=1,3$. Let us now consider the case $i=2$. If none of $x$, $y$ is of the form $(aH_i,\alpha^\prime_i(a,1)b_i)$, then we can take $K=2C$ by (\ref{rixy}). Further, we can assume that the first line in the definition of $\phi_2$ applies to at most one of $x$, $y$ (otherwise Case 1 applies). Thus it suffices to consider the case $u=\alpha^\prime_i(a,1)b_i$, $\phi_2(aH_i,u)= (aH_i, \alpha^\prime_i(a,1)b_i^\prime)$ and $\phi_2(aH_i, v)=(aH_i, \rho_i(v))$. Combining (\ref{ribi}), (\ref{rihbi}), and (\ref{rixy}) we obtain
\begin{align*}
\d_3(\phi_2(x),\phi_2(y))& \le  \d_{R_i^\prime} (\alpha^\prime_i(a,1)b_i^\prime, \rho_i(v))= \d_{R_i^\prime} (\alpha^\prime_i(a,1)\rho_i(b_i), \rho_i(v))
\\ &
\le \d_{R_i^\prime} (\alpha^\prime_i(a,1)\rho_i(b_i), \rho_i(u))+ \d_{R_i^\prime} (\rho_i(u), \rho_i(v))\le 3C.
\end{align*}
\end{proof}

Let us now return to the proof of Proposition \ref{prop-ind}. It is easy to see that $\phi_1$ and $\phi_3$ are bijective on vertex sets and the inverse maps are obtained by reversing the roles of $\mathcal T$ and $\mathcal T^\prime$ (respectively $X$ and $X^\prime$) in the construction. Thus Lemmas \ref{fGeq} and \ref{fLip} apply to $\phi_1^{-1}$ and $\phi_3^{-1}$ as well. This easily implies that $\phi_1$ and $\phi_3$ are quasi-isometries.

Further, recall that a map $f\colon R\to S$ between two metric spaces is called \emph{coarsely surjective} if there exists a constant $\e$ such that $S$ coincides with the closed $\e$-neighborhood of $f(R)$. The map $\phi_2$ is coarsely surjective because so are all $\rho_i$. To find a coarse inverse of $\phi_2$, choose a coarsely $H_i$-invariant coarse inverse $\rho^\prime_i$ of each $\rho_i$ (see Definition \ref{defn:cinv} and Lemma \ref{lem:cinv}). The map $\phi_{2}^\prime$ defined in the same way as $\phi_2$ but reversing the roles of $(\mathcal{A},\mathcal{B})$ and $(\mathcal A^\prime,\mathcal B^\prime)$ and using the quasi-isometries $\rho^\prime_i$, we see that $\phi_2^\prime(\phi_2(y))=y$ whenever $y=a(H_i,b_i)$ for some $a\in G$, otherwise for $y=(H_i,v)$ we have
\[
 \d_S(y,\phi_2^\prime(\phi_2(y))) \le \d_{R_i}(r,\rho^\prime_i(\rho_i(r)))
\]
which is uniformly bounded. Thus $\phi_2^\prime$ is indeed a (coarsely $G$-equivariant) coarse inverse of $\phi _2$. It is straightforward to check that the existence of such a map implies that $\phi _2$ is also a quasi-isometry.

To complete the proof of the proposition, it remains to note that for every group $G$ acting on graphs $R$, $S$, every coarsely $G$-equivariant quasi-isometry $V(R)\to V(S)$ can be extended to a coarsely $G$-equivariant quasi-isometry $R\to S$. Thus the actions of $G$ on the spaces $S_{X,\mathcal T,\mathcal B,\mathcal A}$, $S_{X,\mathcal T^\prime,\mathcal B,\mathcal A}$, $S_{X,\mathcal T^\prime,\mathcal B^\prime,\mathcal A^\prime}$, and $S_{X^\prime,\mathcal T^\prime,\mathcal B^\prime,\mathcal A^\prime}$ are equivalent.
\end{proof}

Recall that $\AG$ denotes the set of all equivalence classes of actions of a group $G$ on geodesic metric spaces (of cardinality at most $c$). Proposition \ref{prop-ind} allows us to formulate the following.

\begin{defn}[Induced action]
Let
\begin{equation}\label{actH}
A=( [H_1\curvearrowright R_1], \ldots, [H_n\curvearrowright R_n] ) \in \AHi.
\end{equation}
By Proposition \ref{graph}, we can assume that every $R_i$ is a graph and the action of $H_i$ on $V(R_i)$ is free.  We define the \emph{induced action} $\Ind_X(A)\in \AG$ by the formula
$$\Ind_X(A)=[\Ind_{X,\mathcal T,\mathcal B}(\mathcal A)],$$
where $\mathcal T$ is any transversal of $G$ with respect to $\{ H_1, \ldots, H_n\}$, $\mathcal B=\{ b_1, \ldots, b_n\}$ is any collection of base vertices $b_i\in V(R_i)$, and
$$
\mathcal A=\{ H_1\curvearrowright R_1, \ldots, H_n\curvearrowright R_n\}.
$$

In the situation where $G$ is finitely generated relative to $H_1,\ldots,H_n$ the induced action does not depend on the choice of finite relative generating set by Proposition \ref{prop-ind}, so we may also define a map $$\Ind\colon\AHi\to\AG$$ by the formula
$$\Ind(A)=[\Ind_{X,\mathcal T,\mathcal B}(\mathcal A)],$$
for every $A$ as in (\ref{actH}), where $X$ is any finite relative generating set.
\end{defn}

Notice that properness and coboundedness of a group action of a metric space is invariant under equivalence. Thus it makes sense to define proper and cobounded elements of $\AG$. The following proposition summarizes some elementary properties of the induced action, which follow immediately from our construction.

\begin{prop}\label{proper}
Let $G$ be a group, $\Hi$ a collection of subgroups of $G$, $X$ a generating set of $G$ relative to $\Hi$. Let $A=( A_1, \ldots, A_n )\in \AHi$.
\begin{enumerate}
\item[(a)] Assume that $A_i$ is cobounded for all $i$. Then so is $\Ind_X(A)$.
\item[(b)] Suppose that $H_i$ is generated by a set $Y_i$ and let $Y=\bigcup_{i=1}^n Y_i$. If $A_i=[H_i\curvearrowright \Gamma (H_i, Y_i)]$, then $Ind_X(A) =[G \curvearrowright \Gamma (G, X\cup Y)]$.
\item[(c)] If $G$ is finitely generated relative to $\Hi$ and $A_i$ is proper for all $i$, then $\Ind(A)$ is proper.
\end{enumerate}
\end{prop}

In order to better understand the construction of the induced actions, we also recommend the reader to consider the following.

\begin{ex}
Let $G$ be the fundamental group of a finite graph of groups with vertex groups $\{ G_v\}_{v\in V}$. Let $A=(A_v)_{v\in V}$ be the collection of equivalence classes of trivial actions $A_v=[G_v\curvearrowright \{ pt\}]$. Then $G$ is finitely generated relative to $\{ G_v\}_{v\in V}$ and $\Ind(A)$ is the equivalence class of the action of $G$ on the associated Bass-Serre tree.
\end{ex}

\subsection{Incompressible subgroups}

Our next goal is to introduce the notion of an incompressible collection of subgroups and to prove Theorem \ref{main3}. The reader is encouraged to review Section 2.4 before reading the following.

\begin{defn}[Incompressible subgroups]
Let $G$ be a group, $\Hl$ a (possibly infinite) collection of subgroups of $G$, and let $X$ be a generating set of $G$ relative to $\Hl$. We say that the collection $\Hl$ is \emph{incompressible} in $G$ with respect to $X$ if for every collection $C=\{ \d_{H_\lambda}\}_{\lambda \in \Lambda}$ of left invariant metrics $\d_{H_\lambda} \in \mathcal M(H_\lambda)$, the inclusion map $H_\lambda \to G$ gives rise to a quasi-isometric embedding $(H_\lambda, \d_{H_\lambda})\to (G, \d_{C,X})$ for every $\lambda \in \Lambda$, where $\d_{C,X}$ is the corresponding induced metric on $G$ (see Definition \ref{IMG}).

Further, if $G$ is finitely generated with respect to $\Hl$, we say that the collection of subgroups $\Hl$ is \emph{incompressible} in $G$ if it is incompressible with respect to some finite generating set $X$ of $G$ relative to $\Hl$. In particular, this definition makes sense if $G$ is finitely generated.
\end{defn}

We could replace ``some finite generating set" with ``any finite generating set" in the definition above. Moreover, we have the following.

\begin{lem}\label{XYinc}
Let $X,Y$ be two generating sets of $G$ with respect to $\Hl$ such that the symmetric difference $X\triangle Y$ is finite. Then $\Hl$ is incompressible with respect to $X$ if and only if it is incompressible with respect to $Y$.
\end{lem}
\begin{proof}
It is easy to see using the definition of the induced metric that for every collection $C=\{ \d_{H_\lambda}\}_{\lambda \in \Lambda}$ of left invariant metrics $\d_{H_\lambda} \in \mathcal M(H_\lambda)$, the identity map $(G, \d_{C,X})\to (G,\d_{C,Y})$ is Lipschitz with the Lipschitz constant $\max_{x\in X}\{ |x|_Y\}$; the maximum exists since $X\triangle Y$ is finite.
\end{proof}

In particular, Lemma \ref{XYinc} holds true if both $X$ and $Y$ are finite.

The main result of this section is the following (Theorem \ref{main3} is clearly a particular case of it).

\begin{thm}\label{EP<->inc}
Let $G$ be a group, $\Hi$ a collection of subgroups of $G$. Suppose that $G$ is finitely generated relative to $\Hi$. Then the following conditions are equivalent.
\begin{enumerate}
\item[(a)] The extension problem for $\Hi$ and $G$ is solvable.
\item[(b)] $\Hi$ is incompressible in $G$.
\item[(c)] For every $A\in \AHi$, $Ind(A)$ is an extension of $A$.
\end{enumerate}
\end{thm}

Using this and Proposition \ref{undist} we see that incompressible subgroups are finitely generated and undistorted. Example \ref{vfree} shows that the converse fails.

We break the proof of Theorem \ref{EP<->inc} into two lemmas.

\begin{lem}\label{EP->inc}
Let $G$ be a group, $\Hi$ a collection of subgroups of $G$. Suppose that $G$ is finitely generated modulo $\Hi$ and the extension problem for $\Hi$ and $G$ is solvable. Then $\Hi$ is incompressible in $G$.
\end{lem}

\begin{proof}
Given a collection $C=\{ \d_{H_1}, \ldots, \d_{H_n}\}$ of metrics $\d_{H_i}\in \mathcal M(H_i)$, let $d_{C,X}$ be the metric on $G$ induced by $C$ and a finite relative generating set $X$ (see Definition \ref{IMG}).  For each $i$, let $\phi_i\colon (H_i,\d_{H_i})\to (G,\d_{C,X})$ be the map induced by inclusion. Our goal is to prove that it is a quasi-isometric embedding. It is clear that $\phi_i$ is Lipschitz. Thus we only need to prove that $\phi _i$ cannot ``compress points too much," i.e., that it satisfies the left inequality in (\ref{defqi}).

Our proof will be based on the following observation, which is straightforward to verify using the definitions: if a composition of two Lipschitz maps is a quasi-isometric embedding, then each of these maps is a quasi-isometric embedding.

Let $G\curvearrowright S$ be an extension of $\{H_1\curvearrowright (H_1, \d_{H_1}), \dots, H_n\curvearrowright (H_n, \d_{H_n})\}$. By the definition of an extension, for every $i$ there is a coarsely $H_i$-equivariant quasi-isometric embedding $\beta_i\colon (H_i,\d_{H_i})\to S$. We fix $s\in S$. Without loss of generality we can assume that $\beta_i(1)=s$ for all $i$.

Assume that each $\beta_i$ satisfies the definition of a quasi-isometric embedding (\ref{defqi}) and the property of being coarsely $H_i$-equivariant (\ref{defce}) with a constant $K$. Since $$\sup_{h\in H_i}\{\d_S(hs,\beta_i(h))\}=\sup_{h\in H_i}\{\d_S(h\beta(1),\beta_i(h1))\}\le K,$$ the orbit map $\mathcal O^s_{H_i}\colon (H_i,\d_{H_i})\to (H_is,\d_S)$ is also a quasi-isometric embedding. It is clear that $$\mathcal O^s_{H_i}= \mathcal O^s_{G}\circ\phi_i.$$ As we remarked above, it suffices to show that $\mathcal O^s_{G}$ is Lipschitz.

Let $f,g\in G$ and let $f_1\dots f_k$ be a geodesic decomposition of $f^{-1}g$; that is, $$f_1, \ldots, f_k\in \left(X\cup\bigcup\limits_{i=1}^n H_{i}\right)\setminus\{1\} \quad  \textrm{and} \quad
\d_{C,X}(f,g)=\sum\limits_{j=1}^k w_{C,X}(f_j).
$$
If $f_j\in H_i$ for some $i$ and $w_{C,X}(f_j)=\d_{H_i}(1, f_j)$, then
$$
\d_S(s,f_js)=\d_S (\beta_i(1), \beta_i(f_j))+ \d_S(\beta_i(f_j), f_j\beta_i(1))\leq  K \d_{H_i}(1,f_j) +2K \le  Kw_{C,X}(f_j)+2K.
$$
As $f_j\neq 1$, it follows that $w_{C,X}(f_j)\geq 1$, and thus $\d_S(s,f_js)\leq 3K w_{C,X}(f_j)$.
Further, if $f_j\in X$, then $\d_S(s,f_js)\leq M=\max_{x\in X} \d_S(s, xs)$. Hence, we have
$$\d_S(fs,gs)\leq \sum_{j=1}^k\d_S(s,f_js)\leq \sum_{j=1}^k (3K+M) w_{C,X}(f_j) = (3K+M)\d_{C,X}(f,g).$$
Thus, $\mathcal O^s_G$ is Lipschitz, as required.
\end{proof}

\begin{lem}\label{inc->ext}
Let $G$ be a group, $\Hi$ a collection of subgroups of $G$, $X$ a (not necessarily finite) generating set of $G$ modulo $\Hi$. Suppose that $\Hi $ is incompressible in $G$ with respect to $X$. Then for every  $A\in \mathcal A(H_1)\times\cdots\times \mathcal A(H_n)$, $\Ind_X(A)$ is an extension of $A$.
\end{lem}

\begin{proof}
Let $A=(A_1, \ldots, A_n)$. By Proposition \ref{graph}, we can choose  $H_i\curvearrowright R_i\in A_i$ for each $i$ so that $R_i$ is a graph and the action of $H_i$ restricted to the vertex set of $R_i$ is free. We fix a transversal $\mathcal T=\{ t_1, \ldots, t_n\}$ and a collection of basepoints $\mathcal B=\{b_1,\ldots,b_n\}$ as in Section 3.3. Let $S=S_{X, \mathcal T, \mathcal B, \mathcal A}$, where $\mathcal A=\{ H_1\curvearrowright R_1, \ldots, H_n\curvearrowright R_n\}$, be the space of the induced action (see Definition \ref{def-ia}).

By Proposition \ref{AABB}, it suffices to prove that, for each $i$, the natural inclusion
\[
\psi_i: V(R_i) \to V(S) \quad \textrm{given by} \quad \psi_i(r)=(H_i,r)
\]
is an $H_i$-equivariant quasi-isometry. The $H_i$-equivariance follows immediately from condition ({\bf A}$_1$) in the definition of the induced action. Note that $\psi_i$ extends to an embedding of $R_i$ as a subgraph of $S$, so each $\psi_i$ is $1$-Lipschitz. Thus we only need to prove that $\psi _i$ satisfies the left inequality in the definition of a quasi-isometric embedding (\ref{defqi}).

We define left invariant metrics $\d_{H_i}$ on subgroups $H_i$ by
\begin{equation}\label{dHidef}
\d_{H_i}(g_1,g_2)=\d_{R_i}(g_1b_i,g_2 b_i)
\end{equation}
for all $g_1, g_2\in H_i$, and set $C=\{\d_{H_1},\ldots,\d_{H_n}\}$.  Since $\Hi$ is incompressible, there exists a constant $D\geq 1$ such that
\begin{equation}\label{dHidCX}
 \d_{H_i}(g_1,g_2) \leq D\d_{C,X}(g_1,g_2)
\end{equation}
for all $g_1, g_2\in H_i$.

We claim that
\begin{equation}\label{dCXdS}
\d_{C,X}(g_1, g_2) \le \d_S(g_1(H_i,b_i),g_2(H_i, b_i))
\end{equation}
for any $g_1, g_2\in G$. (In fact, this is an equality, but the inequality is sufficient for our goal.)

Indeed let $p$ be a geodesic path in the graph $S$ connecting $g_1(H_i,b_i)$ to $g_2(H_i,b_i)$. Let $v_1=g_1(H_i,b_i), v_2, \ldots, v_{k+1}=g_2(H_i,b_i)$ be consecutive vertices of $p$ that belong to $G(H_i,b_i)$, the set of vertices of $\Gamma (G,X)$ considered as a subgraph of $S$. That is, for every $j=1, \ldots, k+1$, we have $v_j=u_j(H_i, b_i)$ for some $u_j\in G$ (recall that these are independent of the choice of $i$, see (\ref{ec})), where $u_1=g_1$ and $u_{k+1}=g_2$. Let $f_j=u_j^{-1}u_{j+1}$. For every $j=1, \ldots, k$, the vertices $v_j, v_{j+1}$ either span an edge of $\Gamma (G,X)$ or simultaneously belong to some $$Z_{j}=\{ u_j H_{i(j)}\} \times R_{i(j)}=\{ u_{j+1} H_{i(j)}\} \times R_{i(j)}.$$ In the former case, we have $f_j\in X$ and therefore
$$
w_{C,X} (f_j)=1=\d_S (v_j,v_{j+1}).
$$
In the latter case, we have $f_j\in H_{i(j)}$ and $v_{j+1}=u_jf_j(H_{i(j)},b_{i(j}))=u_j(H_{i(j)},f_jb_{i(j}))$ by property ($\mathbf A_1$) (see Lemma \ref{act}). Hence
\begin{align*}
w_{C,X} (f_j)& \le \d_{H_{i(j)}}(1, f_j)=\d _{R_{i(j)}}(b_{i(j)}, f_jb_{i(j)})
=\d _{Y_{i(j)}}((H_{i(j)}, b_{i(j)}), (H_{i(j)},f_jb_{i(j)}))
\\ &
= \d _{Y_{i(j)}}(u_j(H_{i(j)}, b_{i(j)}), u_j(H_{i(j)},f_jb_{i(j)}))=\d _{Y_{i(j)}}(v_j,v_{j+1})\le \d_S (v_j, v_{j+1}),
\end{align*}
where $Y_j$ is defined as in Section 3.3, see (\ref{Yi}).
Therefore,
\begin{align*}
\d_{C,X} (g_1, g_2)& = \d_{C,X} (g_1, g_1f_1\cdots f_k) \le \sum\limits_{j=1}^k w_{C,X} (f_j)
\le \sum\limits_{j=1}^k \d_S (v_j, v_{j+1})
\\&
=\d_S(g_1(H_i,b_i),g_2(H_i, b_i)).
\end{align*}
This finishes the proof of (\ref{dCXdS}).

Now let $i$ be fixed and let $r_1,r_2$ be vertices of $R_i$. Let $q$ be a geodesic path from $(H_i,r_1)$ to $(H_i,r_2)$ in $S$. If this path does not intersect $G(H_i,b_i)$, then $\d_{R_i}(r_1,r_2)=\d_S(\psi_i(r_1),\psi_i(r_2))$ and we are done. Otherwise, let $g_1(H_i,b_i)$ and $g_2(H_i,b_i)$ be the first and last vertices in $q\cap G(H_i,b_i)$. Clearly $g_1, g_2\in H_i$. Using  (\ref{dCXdS}), (\ref{dHidCX}), and (\ref{dHidef}) we obtain
\begin{align*}
 \d_S(\psi_i(r_1),\psi_i(r_2))
 &
  = \d_S((H_i,r_1),(H_i,g_1b_i)) + \d_S((H_i,g_1b_i),(H_i,g_2b_i)) + \d_S((H_i,g_2b_i),(H_i,r_2))
 \\
 &
  \geq \d_{R_i}(r_1,g_1b_i) + \d_{C,X}(g_1, g_2) + \d_{R_i}(g_2b_i,r_2)
  \\
 &
  \geq \d_{R_i}(r_1,g_1b_i) + D^{-1} \d_{H_i}(g_1, g_2) + \d_{R_i}(g_2b_i,r_2)
  \\
 &
  = \d_{R_i}(r_1,g_1b_i) + D^{-1} \d_{R_i}(g_1b_i, g_2b_i) + \d_{R_i}(g_2b_i,r_2)
 \\
 &
  \geq D^{-1} \d_{R_i}(r_1,r_2). \qedhere
\end{align*}
\end{proof}

\begin{proof}[Proof of Theorem \ref{EP<->inc}]
That (a) implies (b) follows from Lemma \ref{EP->inc}. Further (b) implies (c) by Lemma \ref{inc->ext}. Finally, (c) implies (a) by Proposition \ref{geo}.
\end{proof}


\section{Proofs of the main results}


\subsection{The extension problem for elementary amenable groups}

Recall that the class of amenable groups is closed under the following four
operations:

\begin{enumerate}
\item[(S)] Taking subgroups.

\item[(Q)] Taking quotient groups.

\item[(E)] Group extensions.

\item[(U)] Directed unions.
\end{enumerate}

As in \cite{Day}, let  $EG$ denote the class of {\it elementary
amenable groups}, that is, the smallest class  which contains all
abelian and finite groups and is closed under the operations
(S)--(U). In particular, $EG$ contains all solvable groups.

To prove Theorem \ref{main1}  we will need several elementary facts. The first one is well-known; we provide a brief outline of the proof for convenience of the reader.

\begin{lem}\label{vab}
Every subgroup of a finitely generated virtually abelian group is undistorted. That is, if $G$ is a virtually abelian group generated by a finite set $X$ and $H\le G$ is generated by a finite set $Y$, then the natural map $(H, \d_Y)\to (G, \d_X)$ is a quasi-isometric embedding.
\end{lem}

\begin{proof}
Let $A\le G$ be a free abelian subgroup of finite index in $G$. It is easy to show that $H$ is undistorted in $G$ if (and only if) $H\cap A$ is undistorted in $A$. The latter result follows from the well-known (and easy to prove) facts that every subgroup of a finitely generated free abelian group $A$ is a retract of a finite index subgroup of $A$ and that retracts and finite index subgroups of finitely generated groups are undistorted.
\end{proof}

The following lemma can be found in \cite{Osi01}.

\begin{lem}\label{EG1}
Let $G$ be an elementary amenable group such that every subgroup of $G$ is finitely generated. Then $G$ is virtually polycyclic.
\end{lem}

The next result is also well-known (see, for example, the proof of Proposition 4.17 in \cite{BH}).

\begin{lem}\label{EG2}
Let $G$ be a virtually polycyclic group. If every subgroup of $G$ is undistorted, then $G$ is virtually abelian.
\end{lem}

We are now ready to prove Theorem \ref{main1}.

\begin{proof}[Proof of Theorem \ref{main1}]
Assume first that $G$ is a finitely generated virtually abelian group that splits as described in Theorem \ref{main1}. We want to show that the extension problem is solvable for all $H\le G$. By Lemma \ref{finind}, we can assume that $H\le A$ without loss of generality. Furthermore, by Theorem \ref{main3} it suffices to show that every $H\le A$ is incompressible in $G$.

Let $X$ be a finite generating set of $G$. Let $C=\{ \d_H\}$, where $\d_H\in \mathcal M(H)$ (see Definition \ref{MG}). Let $w_{C,X}$ and $\d_{C,X}$ denote the corresponding weight function and the induced metric on $G$ (see Definition \ref{IMG}). Finally let $Y$ be a finite generating set of $H$. By Lemma \ref{vab}  there exists a constant $D$ such that
\begin{equation}\label{hYX}
|h|_Y\le D|h|_X
\end{equation}
for every $h\in H$.

Let
\begin{equation}\label{decg}
g=f_1\ldots f_k
\end{equation}
be a geodesic decomposition of an element $g\in H$; that is, we have $f_i\in X\cup H$ for all $i$ and
\begin{equation}\label{dg1}
\d_{C,X}(1,g)=\sum\limits_{i=1}^k w_{C,X}(f_i).
\end{equation}

By our assumption, for every $a\in A$ and every $g\in G$, we have
\begin{equation}\label{ga=ag}
ga=a^{\pm 1}g.
\end{equation}
 In particular, this is true for every $a\in H$ since $H\le A$. We can rearrange the multiples in (\ref{decg}) using (\ref{ga=ag}) so that $f_1, \ldots, f_n\in X\setminus\{ 1\}$ and $f_{n+1}, \ldots, f_k\in H\setminus (X\cup \{1\})$ for some $0\le n\le k$. Let $f=f_1\cdots f_n$ and $h=f_{n+1}\cdots f_k$; we assume here that $f=1$ (respectively, $h=1$) if $n=0$ (respectively, $n=k$).  Thus $g=fh$. Since $g,h\in H$, we have $f\in H$. Since $Y$ is finite, there exists $M=\max_{y\in Y} \d_H(1,y)$. Using (\ref{hYX}) we obtain
$$
\d_H(1,f)\le M|f|_Y\le MD|f|_X\le MDn = MD\sum\limits_{i=1}^n w_{C,X} (f_i)
$$
and
$$
\d_H (1,h)\le\sum\limits_{i=n+1}^k \d_H(1,f_i) = \sum\limits_{i=n+1}^k w_{C,X}(f_i).
$$
Taking these two inequalities together and using (\ref{dg1}), we obtain
$$
\d_H(1,g)\le \d_H(1,f)+\d_H(1,h)\le (MD+1)\sum\limits_{i=1}^k w_{C,X}(f_i)=(MD+1)\d_{C,X}(1,g).
$$
Hence $\d_H(a,b)\le (MD +1)\d_{C,X}(a,b)$ for all $a,b\in H$. This completes the proof of the backward implication in Theorem \ref{main1}.

To prove the forward implication, we first note that if the extension problem is solvable for all $H\le G$, then $G$ must be virtually abelian. Indeed by part (a) of Proposition \ref{undist} and Lemma \ref{EG1} the group $G$ must be virtually polycyclic and hence it is virtually abelian by part (b) of Proposition \ref{undist} and Lemma \ref{EG2}. Let $A\lhd G$ be a finite index free abelian subgroup of $G$. We will first show that (\ref{ga=ag}) holds for all $a\in A$ and $g\in G$. (Note that this is a priori weaker than the assumption that the action of $Q=G/A$ on $A$ factors through the action of $\mathbb Z/2\mathbb Z$ by inversion as the exponent of $a$ in the right side of (\ref{ga=ag}) may depend on $a$.) Since $A$ is free abelian, it suffices to prove (\ref{ga=ag}) under the assumption that $a$ is not a proper power, i.e., $\langle a\rangle$ is a maximal cyclic subgroup of $A$.

Arguing by contradiction, assume that $gag^{-1}=b\ne a^{\pm 1}$ for some $a\in A$ and $g\in G$. If $\langle a,b\rangle $ is cyclic, then $b=a^n$ for some $n\in \mathbb Z$ since $\langle a\rangle$ is a maximal cyclic subgroup of $A$. Hence we have $gag^{-1}=a^n$ and $n\ne \pm 1$. Since $A$ is of finite index in $G$, there exists $k\in \mathbb N$ such that $g^k\in A$ and we have $$ a=g^{k}ag^{-k}=a^{n^k}.$$ This contradicts the assumption that $A$ is torsion free.

Thus the subgroup $H=\langle a,b\rangle $ has rank $2$ and therefore is naturally isomorphic to $\langle a\rangle \oplus\langle b\rangle \cong \mathbb Z^2$. Let $H\curvearrowright \mathbb R$ be the action of $H$ such that $a$ acts trivially and $b$ acts by translation: $bx=x+1$ for all $x\in \mathbb R$. Then for any extension $G\curvearrowright S$ of $H\curvearrowright \mathbb R$, the subgroup $\langle a \rangle $ will have bounded orbits in $S$ while the orbits of $\langle b\rangle $ will be unbounded. However this is impossible as these subgroups are conjugate in $G$. This means that the action $H\curvearrowright \mathbb R$ does not extend to an action of $G$.

Thus we have (\ref{ga=ag}) for all $a\in A$ and $g\in G$. To complete the proof it remains to show that the choice of the exponent in the right side of (\ref{ga=ag}) depends only on $g$. Assume that there are $a_1, a_2\in A\setminus \{ 1\}$ such that $ga_1g^{-1}=a_1$ and $ga_2g^{-1} =a_2^{-1}$. Then $a_1a_2^{-1} =ga_1a_2g^{-1}=(a_1a_2)^{\pm 1}$. However the equality $a_1a_2^{-1} =(a_1a_2)^{\pm 1}$ is impossible for non-trivial elements $a_1$, $a_2$ of a free abelian group. This contradiction completes the proof.
\end{proof}

\subsection{Lipschitz retractions to hyperbolically embedded subgroups}

Our next goal is to show that hyperbolically embedded collections of subgroups are incompressible with respect to suitable relative generating sets. We prove a slightly stronger result, Proposition \ref{propLR}, which seems to be of independent interest and may have other applications.

Throughout this section we fix a group $G$, a (possibly infinite) collection of subgroups $\Hl$ of $G$, and a generating set $X$ of $G$ relative to $\Hl$. Until Proposition \ref{propLR} we \emph{do not} assume that $\Hl$ is hyperbolically embedded in $G$.

\begin{defn}[Equivariant nearest point projection]
For every $g\in G$ and $\lambda\in \Lambda$, let $\pi_\lambda\colon G\to H_\lambda$ be a map satisfying
\begin{equation}\label{npp}
\dxh (g, \pi_\lambda(g))=\min\limits_{h\in H_\lambda}\dxh (g, h).
\end{equation}
Then we call $\pi_\lambda$ a \emph{nearest point projection} of $G$ to $H_\lambda$. If, in addition,  $\pi_\lambda (hg)=h \pi_\lambda(g)$ for all $h\in H_\lambda$ and $g\in G$, we say that $\pi_\lambda$ is an \emph{equivariant nearest point projection}. (Note that ``equivariant" refers to ``$H_\lambda$--equivariant", not ``$G$--equivariant" here.)
\end{defn}

\begin{lem}\label{enpp}
For every $\lambda\in \Lambda$, an equivariant nearest point projection $\pi_\lambda\colon G\to H_\lambda$ exists.
\end{lem}
\begin{proof}
Let $T_\lambda$ be the set of representatives of right cosets of $H_\lambda $ in $G$. For every $t\in T_\lambda$, define $\pi_\lambda(t)$ to be an arbitrary element of $H$ satisfying (\ref{npp}) and for every $g\in H_\lambda t$ define
\begin{equation}\label{defpi}
\pi_\lambda(g)=gt^{-1} \pi_\lambda(t).
\end{equation}
Then we have
$$
\dxh (g, \pi_\lambda(g))= \dxh(g, gt^{-1} \pi_\lambda(t))=\dxh(t, \pi_\lambda(t)).
$$
On the other hand, we obtain
$$
\dxh(g, H_\lambda)=\dxh (tg^{-1}g, tg^{-1}H_\lambda)=\dxh (t, H_\lambda)=\dxh(t, \pi_\lambda(t))
$$
since $tg^{-1}\in H_\lambda $. Therefore, $\dxh (g, \pi_\lambda(g))=\dxh(g, H_\lambda)$, i.e., $\pi_\lambda$ is indeed a nearest point projection. It remains to note that for every $t\in T_\lambda$, $h\in H_\lambda$, and $g\in H_\lambda t$, we have $hg\in H_\lambda t$. Using (\ref{defpi}) we obtain
$$
\pi_\lambda(hg)=hgt^{-1}\pi_\lambda(t)= h\pi_\lambda(g),
$$
i.e., $\pi_\lambda$ is equivariant.
\end{proof}

The main result of this section is the following.

\begin{prop}\label{propLR}
Suppose that $\Hl\h (G,X)$ and let $C=\{ \d_{H_\lambda}\}_{\lambda\in \Lambda}$ be a collection of metrics $\d_{H_\lambda}\in \mathcal M(H_\lambda)$. Let $\d_{C,X}$ denote the corresponding induced metric on $G$. Then for every $\lambda\in\Lambda$, every nearest point projection $\pi_\lambda\colon G\to H_\lambda$ induces a Lipschitz map $(G, \d_{C,X})\to (H_\lambda, \d_{H_\lambda})$.
\end{prop}

\begin{proof}
Throughout this proof, we fix $\lambda \in \Lambda$. Let $K$ be a positive integer satisfying
\begin{equation}\label{K4D}
K\ge  \max\{ \d_{H_\lambda}(1, h) \mid  h\in H_\lambda, \; \dl(1, h)\le 4D\},
\end{equation}
where $D$ is the constant from Lemma \ref{ngon}. Note that the maximum is taken over a finite set since $\Hl\h (G,X)$. We will show that for every $f,g\in G$,
\begin{equation}\label{dfg}
\d_{H_\lambda }(\pi_\lambda(f), \pi_\lambda(g))\le K\d_{C,X}(f,g).
\end{equation}

Let $f^{-1}g=f_1\cdots f_k$ be a geodesic decomposition. That is, $f_1, \ldots, f_k$ are elements of the set
$$Y= X\cup \left(\bigcup\limits_{\lambda\in \Lambda} H_{\lambda}\right)$$  and
\begin{equation}\label{dfg1}
\d_{C,X}(f,g)=\sum\limits_{i=1}^k w_{C,X}(f_i).
\end{equation}
To prove (\ref{dfg}) we first show that
\begin{equation}\label{dfft}
\d_{H_\lambda }(\pi_\lambda(a), \pi_\lambda(at))\le Kw_{C,X}(t)
\end{equation}
for every $a\in G$ and $t\in Y$.

To this end we denote by $u$ (respectively $w$) a geodesic path in $\G$ that goes from $\pi_\lambda(a)$ to $a$ (respectively, from $at$ to $\pi_\lambda(at)$). Without loss of generality, we can assume that $t\ne 1$.  Let $e$ be the edge labeled by an element of $H_\lambda$ that connects $\pi_\lambda(at)$ to $\pi_\lambda(a)$. Finally, let $v$ denote the edge of $\G$ starting at $a$ and labeled by $t$ (see Fig. \ref{F2}).

\begin{figure}
  \centering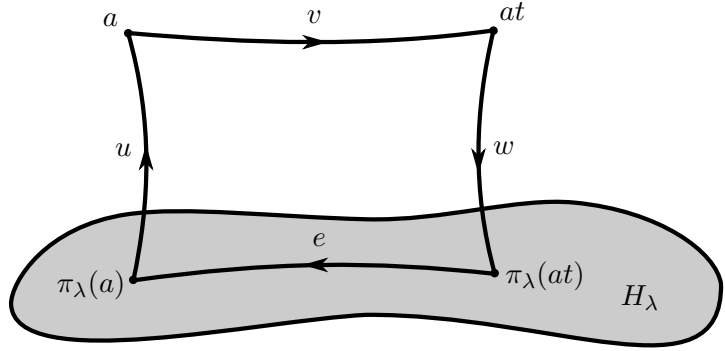\\
  \caption{The proof of Proposition \ref{propLR}.}\label{F2}
\end{figure}

If $e$ is connected to an edge $e^\prime $ of $u$ labeled by an element of $H_\lambda$, then $e^\prime_+\in H_\lambda$ and $\dxh (a, e^\prime_+)< \dxh (a, \pi_\lambda(a))$, which contradicts our assumption that $\pi_\lambda$ is a nearest point projection. Similarly $e$ cannot be connected to an edge of $w$ labeled by an element of $H_\lambda$. Thus $e$ is an isolated $H_\lambda$--component of the path $weu$.

We now consider two cases.

{\it Case 1.} First assume that $e$ is isolated in the geodesic quadrilateral $p=weuv$. Then $$\dl(\pi_\lambda(a), \pi_\lambda(at))=\dl(e_-,e_+)\le 4D$$ by Lemma \ref{ngon}. Therefore, $\d_{H_\lambda }(\pi_\lambda(a), \pi_\lambda(at)) \le K$ by (\ref{K4D}). In particular, (\ref{dfft}) holds since $w_{C,X}$ takes values in $\mathbb N\cup \{ 0\}$ and $w_{C,X}(t)=0$ only if $t=1$.

{\it Case 2.} Suppose now that $e$ is not isolated in $p$. Since $e$ is isolated in $weu$, it can only be connected to $v$. In particular, we have $t\in H_\lambda $. Without loss of generality, we can also assume that
\begin{equation}\label{wct}
w_{C,X}(t)=\d_{H_\lambda}(1,t)
\end{equation}
in this case. Indeed, recall that labels of edges of $\G$ are taken from the \emph{disjoint} union of $X$ and subgroups $H_\mu$, $\mu\in \Lambda$. Thus if we also have $t\in H_\mu $ for some $\mu\ne \lambda$,  we can simply replace $v$ with another edge of $\G$ with the same endpoints as $v$ and the label $t\in H_\mu$, which makes Case 2 impossible. Similarly we rule out the case $t\in X$. Thus, without loss of generality, we can assume that $t\notin X$ and $\Lambda (t)=\{ \lambda\}$ in the notation of Definition \ref{IMG}; hence (\ref{wct}) holds by the definition of $w_{C,X}$.

Since $e$ is connected to $v$, we have $a=v_-\in H_\lambda $ and $at=v_+\in H_\lambda$. Since $\pi_\lambda $ is a nearest point projection, it is identical on $H_\lambda$. In particular, we have $\pi_\lambda(a)=a$ and $\pi_\lambda (at)=at$. Therefore, by (\ref{wct}) we have
$$
\d_{H_\lambda}(\pi_\lambda (a), \pi_\lambda (at)) =\d_{H_\lambda} (a, at)=\d_{H_\lambda} (1, t) =w_{C,X} (t)
$$
in this case. This completes the proof of (\ref{dfft}).

Now let $$h_0=\pi_\lambda (f), \;\; h_1=\pi_\lambda (ff_1), \;\;\ldots, \;\; h_k=\pi_\lambda (ff_1\cdots f_k)=\pi_\lambda(g).$$ Using the triangle inequality, (\ref{dfg1}) and (\ref{dfft}), we obtain
$$
\d_{H_\lambda}(\pi_\lambda (f), \pi_\lambda(g))\le \sum\limits_{i=1}^k \d_{H_\lambda}(h_{i-1}, h_i)\le \sum\limits_{i=1}^k Kw_{C,X}(f_i)=K\d_{C,X}(f,g).
$$
\end{proof}

Since every nearest point projection $G\to H_\lambda$ is the identity on $H_\lambda$, we obtain the following corollary of Lemma \ref{enpp} and Proposition \ref{propLR}.
\begin{cor}\label{corLR}
Assume that $\Hl\h (G,X)$. Then for every collection of left invariant metrics $C=\{ \d_{H_\lambda}\}$, where $\d_{H_\lambda}\in \mathcal M(H_\lambda)$, and every $\lambda \in \Lambda$, there exists an $H_\lambda$--equivariant Lipschitz map $(G, \d_{C,X})\to (H_\lambda, \d_{H_\lambda})$ whose restriction to $H_\lambda$ is the identity map.
\end{cor}

Finally, we pass from the property of being a Lipschitz retract to the property of being undistorted. This is a fairly standard argument.

\begin{cor}\label{he->au}
Suppose that $\Hl\h (G,X)$. Then $\Hl$ is incompressible in $G$ with respect to $X$. In particular, if $G$ is finitely generated and hyperbolic relative to a collection of subgroups $\Hi$, then $\Hi$ is incompressible in $G$.
\end{cor}

\begin{proof}
Fix any $\lambda\in \Lambda$. Let $C=\{ \d_{H_\lambda}\}$, where $\d_{H_\lambda}\in \mathcal M(H_\lambda)$. Let $\pi_\lambda \colon (G, \d_{C,X})\to (H_\lambda, \d_{H_\lambda})$ be a Lipschitz map whose restriction to $H_\lambda$ is the identity map. Let $L$ be the corresponding Lipschitz constant. Then for every $g,h\in H_\lambda$, we have
$$
\d_{H_\lambda}(g,h)= \d_{H_\lambda}(\pi_\lambda (g), \pi_\lambda (h))\le L \d_{C,X}(g, h).
$$
The opposite inequality $\d_{C,X}(g,h)\le \d_{H_\lambda}(g,h)$ for all $g,h\in H$ is obvious from the definition of the induced metric. Thus the inclusion $H_\lambda\to G$  induces a quasi-isometric embedding $(H_\lambda, \d_{H_\lambda})\to (G, \d_{C,X})$. To derive the claim about relatively hyperbolic groups we only need to refer to  Proposition \ref{relhypdef}.
\end{proof}

\subsection{Extending actions of hyperbolically embedded subgroups}

Recall that an equivalence class $A\in \AG$ is called {\it hyperbolic} if it consists of $G$-actions on hyperbolic metric spaces. The main goal of this section is to prove the following result; Theorem \ref{main2} follows from it immediately via Proposition \ref{relhypdef}.

\begin{thm}\label{he}
Let $G$ be a group, $\Hi$ a collection of subgroups of $G$, and let $X$ a be relative generating set of $G$ with respect to $\Hi$. \begin{enumerate}
\item[(a)] Suppose that $\Hi\h (G,X)$. Then for every collection
\begin{equation}\label{a1an}
A=(A_1, \ldots, A_n)\in \AHi,
\end{equation}
the induced action $\Ind_X(A)$ is an extension of $A$; if, in addition, each $A_i$ is hyperbolic, then so is $\Ind_X(A)$.
\item[(b)] Conversely, suppose that $H_1, \ldots, H_n$ are countable and for every collection (\ref{a1an}), where each $A_i$ is hyperbolic, the induced action $\Ind_X(A)$ is a hyperbolic extension of $A$. Then $\Hi\h (G,X)$.
\end{enumerate}
\end{thm}

\begin{proof}
(a) That $\Ind_X(A)$ is an extension of $A$ follows from Corollary \ref{he->au} and Lemma \ref{inc->ext}. Thus we only need to prove hyperbolicity. The proof is fairly standard; it essentially repeats the proof of \cite[Lemma 6.45]{DGO} with obvious adjustments.

For details about van Kampen diagrams, isoperimetric functions, etc., we refer to Section 2.2 and \cite{DGO}. Given a (combinatorial) path $p$ in a van Kampen diagram $\Delta $ over a group presentation, we denote by $\Lab (p)$ the label of $p$.

Theorem \ref{ipchar} provides us with a strongly bounded relative presentation
\begin{equation}
G=\langle X, \mathcal H  \mid  \mathcal S\cup \mathcal R\rangle
\label{rp2}
\end{equation}
with linear relative isoperimetric  function.

Let $\mathcal A=\{H_1\curvearrowright R_1,\ldots,H_n\curvearrowright R_n\}$, where for each $i$ we have $H_i\curvearrowright R_i\in A_i$, $R_i$ is a hyperbolic graph, and the action of $H_i$ restricted to the vertex set of $R_i$ is free. We fix some transversal $\mathcal T$ and collection of base points $\mathcal B$ as in Section 3.3 and let $G\curvearrowright S=\Ind_{X,\mathcal T,\mathcal B}(\mathcal A)\in \Ind_{X}(A)$. In what follows we naturally think of $\Gamma (G,X)$ as a subgraph of both $\Gamma (G, X\sqcup \mathcal H)$ and $S$.

\begin{figure}
  \centering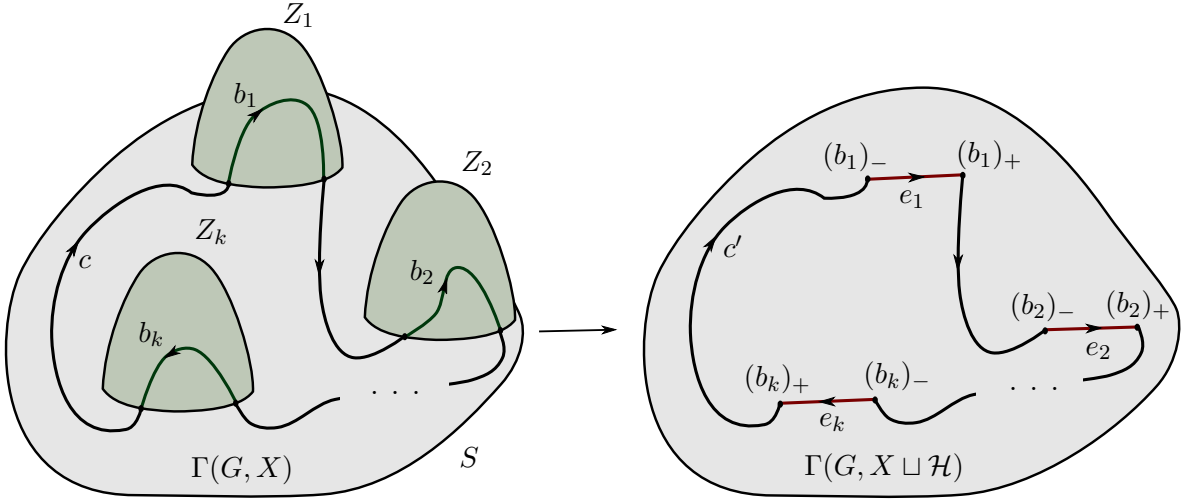\\
  \caption{Constructing the loop $c^\prime$}
\end{figure}

We will show that condition (b) from Proposition \ref{IP} holds for $S$. Let $c$ be a loop in $S$. Without loss of generality we can assume that $c$ has at least one vertex in $\Gamma (G,X)$ as otherwise it is contained in a copy of a hyperbolic graph $R_i$ and the isoperimetric inequality follows. (Notice that we need finiteness of the collection $\Hi$ here to ensure uniformness of the isoperimetric constants.)

To every such $c$ we associate a loop in $\G$  as follows. Let $b_1, \ldots , b_k$ be the set of all maximal subpaths of $c$ such that each $b_i$ belongs to some $\{ g_iH_{j(i)}\} \times R_{j(i)}$. We replace each $b_i$ with the corresponding edge $e_i$ in $\G$ connecting $(b_i)_-$ to $(b_i)_+$ and labeled by an element of $H_{j(i)}$. This naturally defines a loop $c^\prime $ in $\G$.

Consider a van Kampen diagram $\Delta $ over (\ref{rp2}) such that:
\begin{enumerate}
\item[1)] The boundary label of $\Delta $ is $\Lab (c^\prime)$.
\item[2)] $\Delta $ has minimal number of $\mathcal R$-cells among all diagrams satisfying 1).
\item[3)] Every edge of $\partial \Delta $ labeled by a letter from $\mathcal H$ belongs to an $\mathcal S$-cell.
\item[4)] $\Delta $ has minimal number of $\mathcal S$-cells among all diagrams satisfying 1)--3).
\end{enumerate}
Note that we can always ensure 3) by gluing cells labeled by $hh^{-1}$ where $h\in \mathcal H$ (the so-called $0$-cells) to the boundary of $\Delta$. In what follows we identify $\partial \Delta $ with $c^\prime$.

The maps $e_i\mapsto b_i$ naturally induce a continuous map $\phi $ from $c^\prime$ to $S$ whose image is $c$. Observe that 4) implies that no edge of $\Delta $ can belong to two $\mathcal S$-cells, for otherwise these $\mathcal S$-cells could be replaced with a single cell. Thus every internal edge $e$ of $\Delta $ belongs to an $\mathcal R$-cell and hence $\Lab (e)\in X\cup Y$, where $Y$ is the set of all letters from $\mathcal H$ that appear in relations $R\in \mathcal R$. Since (\ref{rp2}) is strongly bounded, $Y$ is finite. By Proposition \ref{prop-ind} and Lemma \ref{rhX}, we can assume without loss of generality that for every $y\in Y$ there exists $x_y\in X$ such that $x_y$ and $y$ represent the same element of $G$. This allows us to extend $\phi $ to the $1$-skeleton of $\Delta $ by mapping every internal edge $e$ of $\Delta $ to the corresponding edge of $\Gamma (G, X)\subseteq S$ (edges labeled by $y\in Y$ are mapped to the corresponding edges labeled by $x_y\in X$).

Let $f(n)=Cn$ be a relative isoperimetric function of (\ref{rp2}) and let $M=\max_{R\in \mathcal R} \| R\| $. Note that $M<\infty $ as (\ref{rp2}) is bounded.  The map $\phi\colon Sk^{(1)} (\Delta )\to S$ gives rise to a decomposition of $[c]$ into the sum of at most $C\ell (c)$ homotopy classes of loops of length at most $M$ corresponding to $\mathcal R$-cells of $\Delta $ (here we use the fact that no $e_i$ belongs to the boundary of an $\mathcal R$-cell, which is ensured by 3)) plus $[s_1]+\cdots + [s_m]$, where all $s_i$ are images of boundaries of $\mathcal S$-cells of $\Delta $. Clearly the total length of all loops $[s_i]$ does not exceed $2$ times the total number of internal edges of $\Delta $ plus $\ell (c)$.  Again taking into account that every internal edge $e$ of $\Delta $ belongs to an $\mathcal R$-cell and using the isoperimetric inequality we obtain
\begin{equation}\label{tlsi}
\sum\limits_{i=1}^m \ell (s_i)\le 2MC\ell (c^\prime )+\ell (c)\le (2MC+1)\ell (c).
\end{equation}

Note that every $s_i$ is a loop in $\{ g_iH_{j(i)}\} \times R_{j(i)}$ for some $g_i\in G$ and $j(i)\in \{ 1, \ldots, n\}$, which is an isometric copy of the (hyperbolic) graph $R_{j(i)}$. Therefore there exist constants $A,B$ such that every $[s_i]$ decomposes into the sum of at most $A\ell (s_i)$ homotopy classes of loops of length at most $B$. Consequently $[c]$ decomposes into the sum of at most $(C+A(2MC+1))\ell(c)$ homotopy classes of loops of length at most $\max \{ M, B\}$.  This completes the proof of (a).

(b)Applying Proposition \ref{proper} to the collection $A=(A_1, \ldots, A_n)$, where $A_i=[H_i\curvearrowright \Gamma (H_i, H_i)]$, we obtain that $\Ind_X(A)=[G\curvearrowright\G]$. It follows from our assumption that the relative Cayley graph $\G$ is hyperbolic. Thus it remains to verify condition (b) from Definition \ref{he-def}.

Arguing by contradiction, assume that condition (b) from the definition of a hyperbolically embedded collection of subgroups does not hold. We are going to construct a specific collection of actions of subgroups $H_i$ on hyperbolic spaces such that the corresponding induced action of $G$ is not a hyperbolic extension.

By the Higman-Neumann-Neumann theorem, every countable group embeds in a finitely generated group. We embed $H_i$ into a finitely generated group $K_i$. Let $Z_i$ be a finite generating set of $K_i$ and let $|\cdot |_{Z_i}$ denote the corresponding word length. That is, for some $j\in \{ 1, \ldots, n\}$, some $C\in \mathbb N$, and every $N\in \mathbb N$, there exist a path $q$ of length at most $C$ in $$\Delta_j = \G \setminus E(\Gamma (H_j, H_j))$$ connecting $1$ to a vertex $h\in H_j$ such that $|h|_{Z_j}\ge N$. Let $R_i= \mathcal H(\Gamma (K_i,Z_i))$ for all $i$ (see Definition \ref{GMhoro}) and let $\mathcal A=(H_1\curvearrowright R_1, \ldots, H_n\curvearrowright R_n)$.

We fix some transversal $\mathcal T$ and a collection of base points $\mathcal B$ as in Section 3.3. By Lemma \ref{GM}, all $R_i$ are hyperbolic. Therefore, so is the space $S=S_{X,\mathcal T,\mathcal B, \mathcal A}$ of the induced action.
Let $u =qr$, where $r$ is the edge of $\G $ labeled by $h^{-1}\in H_j$ and connecting $h$ to $1$. Then $r$ is an isolated $H_j$-component of $u$. Let $v$ be a loop in $S$ obtained from $u$ by replacing all $H_i$-components of $u$ (for all $i$) with geodesics in corresponding graphs $\{ gH_i\} \times R_i$. We call these subpaths of $v$ \emph{$R_i$-components}. Since $\Ind_X(A)$ is an extension of $A$, geodesics in the graphs $\{ gH_i\} \times R_i$ are $(\lambda, c)$-quasi-geodesics  in $S$ for some $\lambda\ge 1$, $c>0$.

Let $s$ denote the subpath of $v$ corresponding to the $H_j$-component $r$ of $u$; thus $s$ is a geodesic in $\{1H_j\} \times R_j$ and therefore it is a $(\lambda,c)$-quasi-geodesic in $S$. If $N$ is large enough compared to $C$, the hyperbolicity constant $\delta$ of $S$, and the constant $\kappa =\kappa (\delta, \lambda, c)$ provided by Lemma \ref{qg}, the combination of Lemma \ref{qg} and Lemma \ref{Ols} provide us with a subpath $s_0$ of $s$ and a subpath $t_0$ of some other $R_i$-component $t$ of $v$ such that $s_0$ and $t_0$ belong to $(2\kappa +15\delta)$-neighborhoods of each other and $\ell (s_0) \ge 2(2\kappa +15\delta) +4$. Note that $t$ cannot belong to $\{1H_j\} \times R_j$ as $r$ is an isolated $H_j$-component of $u$. Since $s_0$ is a geodesic in $\{ 1H_j\} \times R_j$, which is an isometric copy of the combinatorial horoball $R_j$, it must contain a vertical subsegment of length greater than $2\kappa +15\delta$ by Lemma \ref{GM}. Thus $s_0$ cannot belong to the closed $(2\kappa +15\delta)$-neighborhood of $\Gamma (G,X)$. This contradicts the fact that $s_0$ belongs to a $(2\kappa +15\delta)$-neighborhood of $t_0$. Indeed $t_0$ belongs to some $\{gH_i\}\times R_i\ne \{ 1H_j\} \times R_j$ and every path in $S$ originating in $\{ 1H_j\} \times R_j$ and terminating in $\{gH_i\}\times R_i$ must intersect $\Gamma(G,X)$. This contradiction completes the proof of part (b) of Definition \ref{he-def} and the theorem.
\end{proof}

The following immediate corollary is a generalization of Theorem \ref{main2}.

\begin{cor}\label{cor-he}
Let $G$ be a group, $\Hi$ a collection of hyperbolically embedded subgroups of $G$. Then the extension problem for $\Hi$ and $G$ is solvable.
\end{cor}

Yet another corollary follows immediately from Theorem \ref{he} and Proposition \ref{relhypdef}.

\begin{cor}
Let $G$ be a group, $\Hi$ a collection of subgroups of $G$.
\begin{enumerate}
\item[(a)] Suppose that $G$ is hyperbolic relative to $\Hi$. Then for every collection (\ref{a1an}), the induced action $\Ind(A)$ is an extension of $A$ and if each $A_i$ is hyperbolic, then so is $\Ind(A)$.
\item[(b)] Conversely, suppose that $G$ is finitely generated relative to $\Hi$, the subgroups $H_1, \ldots, H_n$  are countable, and for every collection (\ref{a1an}), where each $A_i$ is hyperbolic,  $\Ind(A)$ is a hyperbolic extension of $A$. Then $G$ is hyperbolic relative to $\Hi$.
\end{enumerate}
\end{cor}

We consider a couple of examples illustrating that part (b) of Theorem \ref{main4} can fail under certain weaker assumptions.

\begin{ex}\label{ex-c}
Let $H=Sym(\mathbb N)$ and let $G=H\times \mathbb Z/2\mathbb Z$. As we already mentioned, all actions of $H$ on metric spaces have bounded orbits. This easily implies that for every $A\in \AH$, $\Ind(A)$ is an extension of $A$. In addition, if $[H\curvearrowright R]$ is hyperbolic then the space of the induced action is quasi-isometric to two copies of $R$ glued along a bounded subset; in particular, this space is also hyperbolic. Thus all assumptions of part (b) of Theorem \ref{he} are satisfied except countability of $H$. However the conclusion fails: $G$ is not hyperbolic relative to $H$ as peripheral subgroups of relatively hyperbolic groups must be almost malnormal.
\end{ex}

\begin{ex}\label{ex-d}
Let $G=\mathbb Z \times \mathbb Z/2\mathbb Z$. Then $H=\mathbb Z \times \{ 1\}\le G$ is a retract of $G$, hence every action of $H$ on a hyperbolic space extends to an action of $G$ on the same hyperbolic space, see Example  \ref{ex2} (b). However $G$ is not hyperbolic relative to $H$. This shows that the condition that $\Ind (A)$ is a hyperbolic extension of $A$ for every hyperbolic $A\in \AH$ cannot be replaced with the assumption that every action of $H$ on a hyperbolic space extends to an action of $G$ on a hyperbolic space. More specifically, when the action $A$ is of $\mathbb Z$ on its combinatorial horoball, the action $\Ind(A)$ is on the Cayley graph of $G$ where we attach a combinatorial horoball onto each coset of $\mathbb Z$. This space is not hyperbolic.
\end{ex}

We now turn to the proof of Corollary \ref{tfhyp} from the introduction. Recall that two elements $a, b$ of infinite order of a group $H$ are called \emph{commensurable} in $H$ if some non-trivial powers of $a$ and $b$ are conjugate in $H$. We will need the following.

\begin{defn}
We say that a group embedding $H\le G$ is \emph{commensurability preserving} if infinite order elements of $H$ are commensurable in $H$ whenever they are commensurable in $G$.
\end{defn}

\begin{ex}
If $H$ is almost malnormal in $G$, then the embedding $H\le G$ is commensurability preserving. Indeed suppose that $a, b\in H$ are infinite order elements and are commensurable in $G$. Then there exists $t\in G$ and $m,n\in \mathbb N$ such that $t^{-1}a^mt=b^{\pm n}$. In particular, the intersection $H\cap t^{-1}Ht$ contains $\langle b^n\rangle$ and therefore it is infinite. By almost malnormality, we get $t\in H$, which means that $a$ and $b$ are commensurable in $H$. Note, however, that malnormality is a strictly stronger condition. (Hint: consider the embedding $2\mathbb Z\le \mathbb Z$.)
\end{ex}

We will need two more lemmas.

\begin{lem}\label{lemhyp}
Let $H$ be a subgroup of a hyperbolic group $G$ and let $a,b\in H$ be two non-commensurable (in $H$) elements of infinite order. Then there exists an action of $H$ on a metric space such that the orbits of $\langle a\rangle $ are bounded while the orbits of $\langle b\rangle $ are unbounded.
\end{lem}

\begin{proof}
By \cite[Theorem 6.8]{DGO} applied to the action of $H$ on the Cayley graph of $G$ with respect to a finite generating set, $a$ and $b$ are contained in virtually cyclic subgroups $E(a)$, $E(b)$ of $H$ such that $\{ E(a), E(b)\} \h (H, X)$ for some $X\subseteq H$. Let $A\in \mathcal A(E(a))$ be the equivalence class of a geometric action of $E(a)$ and let $B\in \mathcal A(E(b))$ be the the equivalence class of the trivial action on the point. By Theorem \ref{he}, there is an extension of the pair $(A, B)$ to an $H$--action $C\in \AH$. Clearly $C$ satisfies the required conditions.
\end{proof}

\begin{lem}\label{vc}
Let $G$ be a virtually cyclic group. Then the extension problem is solvable for all subgroups of $G$.
\end{lem}

\begin{proof}
Without loss of generality we can assume that $G$ is infinite. Let $A\cong \mathbb Z$  be a normal cyclic subgroup of finite index in $G$. It is clear that Theorem \ref{main1} applies to $G$ since $Aut(\mathbb Z)\cong \mathbb Z/2\mathbb Z$. Thus the extension problem is solvable for all subgroups of $G$.
\end{proof}

Corollary \ref{tfhyp} is a simplified version of the following.

\begin{cor}\label{hyp}
Let $G$ be a hyperbolic group.
\begin{enumerate}
\item[(a)] Suppose that $H$ is quasiconvex in $G$ and either virtually cyclic or almost malnormal. Then the extension problem is solvable for $H\le G$.
\item[(b)] Conversely, if the extension problem is solvable for a subgroup $H\le G$, then $H$ is quasiconvex and the embedding $H\le G$ is commensurability preserving.
\end{enumerate}
\end{cor}

\begin{proof}
(a) Assume first that $H$ is quasiconvex and almost malnormal. Then by a result of Bowditch \cite{Bow}, $G$ is hyperbolic relative to $H$ and hence the extension problem for $H\le G$ is solvable by Theorem \ref{EP<->inc}. Now assume that $H$ is virtually cyclic. If $H$ is finite the claim is obvious, so we assume that $H$ is infinite. Then there exists a maximal virtually cyclic subgroup $E$ of $G$ containing $H$. By Lemma \ref{vc}, every $H$-action on a metric space extends to an $E$-action. By \cite[Theorem 6.8]{DGO} we have $E\h G$ and therefore every $E$-action in turn extends to a $G$--action. Thus the extension problem is solvable for $H\le G$ in this case as well.

(b) Assume that the extension problem is solvable for $H$. Then $H$ is finitely generated and undistorted in $G$ by Proposition \ref{undist}. This is well-known to be equivalent to quasiconvexity.

To prove that the embedding $H\le G$ is commensurability preserving we argue by contradiction. Assume that  there are elements $a,b\in H$ of infinite order that are commensurable in $G$ but not in $H$. By Lemma \ref{lemhyp}, there exists an $H$--action such that the orbits of $\langle a\rangle $ are bounded while the orbits of $\langle b\rangle$ are not. Since the extension problem is solvable for $H\le G$, there is an action of $G$ with the same property. However, this contradicts the fact that $a$ and $b$ are commensurable in $G$.
\end{proof}

The following example shows that the sufficient condition for the extension problem to be solvable from part (a) is not necessary.

\begin{ex}\label{HxK}
Let $G=H\times K$, where $K$ is a non-trivial finite group and $H$ is a finitely generated non-cyclic free group. Then $G$ is hyperbolic and $H$ is neither virtually cyclic nor malnormal. However the extension problem for $H\le G$ is solvable since $H$ is a retract of $G$.
\end{ex}

We conclude with few open problems.

\begin{prob}
Is the sufficient condition for the extension problem to be solvable from Corollary \ref{hyp} (a) necessary in case $G$ is torsion free?
\end{prob}

The negative answer can likely be obtained by studying the case of $G=F(a,b)$, the free group with basis $\{ a, b\}$, and  $H=\langle a^2, b^2\rangle \le G$.

\begin{prob}
Is the necessary condition for the extension problem to be solvable from Corollary \ref{hyp} (b) sufficient?
\end{prob}

It would be also interesting to address the extension problem for individual subgroups in other classes of groups.

\begin{prob}
Describe incompressible subgroups of finitely generated nilpotent groups, polycyclic groups, free solvable groups, etc.
\end{prob}

\vspace{1cm}

\noindent \textbf{Carolyn Abbott: } Department of Mathematics, University of California, Berkeley, Berkeley 94720, USA. \\
E-mail: \emph{c\underline{\hspace{.35em}}abbott@math.berkeley.edu}

\bigskip

\noindent \textbf{David Hume: } Mathematical Institute, University of Oxford, Woodstock Road, Oxford
OX2 6GG, UK.\\
E-mail: \emph{david.hume@maths.ox.ac.uk}

\bigskip

\noindent \textbf{Denis Osin: } Department of Mathematics, Vanderbilt University, Nashville 37240, USA.\\
E-mail: \emph{denis.v.osin@vanderbilt.edu}

\end{document}